\documentclass[11pt,openany,leqno]{article}
\usepackage{amsmath,amsthm,amsfonts,amssymb,amscd,url}
\usepackage{graphics}
\usepackage{float} 
\input{xy} \xyoption{all}
\usepackage[latin1]{inputenc}

\usepackage{enumerate}
\usepackage{hyperref}
\usepackage{epsfig} 
\input{xy} \xyoption{all}

\makeindex
\begin{document}
\def\Diff{\text{Diff}}
\def\Max{\text{max}}
\def\P{\mathbb P}
\def\R{\mathbb R}
\def\M{\mathbb M}
\def\N{\mathbb N}
\def\Z{\mathbb Z}
\def\a{{\underline a}}
\def\b{{\underline b}}
\def\c{{\underline c}}
\def\Log{\text{log}}
\def\loc{\text{loc}}
\def\inta{\text{int }}
\def\det{\text{det}}
\def\exp{\text{exp}}
\def\Re{\text{Re}}
\def\lip{\text{Lip}}
\def\leb{\text{Leb}}
\def\dom{\text{Dom}}
\def\diam{\text{diam}\:}
\def é{\'e}
\def\supp{\text{supp}\:}
\newcommand{\ovfork}{{\overline{\pitchfork}}}
\newcommand{\ovforki}{{\overline{\pitchfork}_{I}}}
\newcommand{\Tfork}{{\cap\!\!\!\!^\mathrm{T}}}
\newcommand{\whforki}{{\widehat{\pitchfork}_{I}}}
\newcommand{\marginal}[1]{\marginpar{{\scriptsize {#1}}}}
\def\st{{\mathfrak t}}
\def\sT{{\mathfrak T}}
\def\sR{{\mathfrak R}}
\def\sM{{\mathfrak M}}
\def\sA{{\mathfrak A}}
\def\sB{{\mathfrak B}}
\def\sY{{\mathfrak Y}}
\def\sE{{\mathfrak E}}
\def\sP{{\mathfrak P}}
\def\sG{{\mathfrak G}}
\def\sa{{\mathfrak a}}
\def\sw{{\mathfrak w}}
\def\se{{\mathfrak e}}
\def\sb{{\mathfrak b}}
\def\sc{{\mathfrak c}}
\def\sg{{\mathfrak g}}
\def\sd{{\mathfrak d}}

\theoremstyle{plain}
\newtheorem{theo}{\bf Theorem}[section]
\newtheorem{lemm}[theo]{\bf Lemma}
\newtheorem{ques}[theo]{\bf Question}
\newtheorem{sublemm}[theo]{\bf Sublemma}
\newtheorem{IH}[theo]{\bf Extra induction hypothesis}
\newtheorem{prop}[theo]{\bf Proposition}
\newtheorem{coro}[theo]{\bf Corollary}
\newtheorem{Property}[theo]{\bf Property}
\newtheorem{Claim}[theo]{\bf Claim}
\theoremstyle{remark}
\newtheorem{rema}[theo]{\bf Remark}
\newtheorem{important rema}[theo]{\bf Important remark}
\newtheorem{remas}[theo]{\bf Remarks}
\newtheorem{fact}[theo]{\bf Fact}
\newtheorem{partial answer}[theo]{\bf Partial answer}
\newtheorem{exem}[theo]{\bf Example}
\newtheorem{Examples}[theo]{\bf Examples}
\newtheorem{defi}[theo]{\bf Definition}


\title{Normal forms and Misiurewicz renormalization for dissipative surface diffeomorphisms}

\author{Pierre Berger\footnote{CNRS-Université Paris 13, Sorbonne Paris Cité, LAGA, partially financed by the Balzan project of J. Palis and the Brazilian-French Network in Mathematics.}}

\date{\today}
\maketitle
\abstract{We define a hyperbolic renormalizations suitable for maps of small determinant, with uniform bounds for large periods. The techniques involve an improvement of the celebrated Palis-Takens renormalization and normal forms (fibered linearizations). These techniques are useful to study the dynamics of Hénon like maps and the geometry of their parameter space.}

\section*{Introduction}

A key event in chaos theory was the discovery by Lorenz \cite{Lo63}  that an ordinary differential equation modeling a convection flow had most of its orbits which are unstable and non-periodic. This was latter simplified to a mathematical paradigm  by Hénon \cite{He}, as the dynamics given by the iteration of the following diffeomorphism:
\[H\colon (x,y)\mapsto (1-1.4x^2+y,0.3x).\]
He conjectured that this map has a strange attractor.
This conjecture remains open, despite intensive works on the Hénon family
 \[h_{ab}\colon (x,y)\mapsto (x^2+y+a,-bx),\quad a,b\in \R\]
 which is a family of maps of $\R^2$ containing a conjugate of $H$. For instance a celebrated theorem of Benedicks-Carleson \cite{BC2} shows that for a Lebesgue positive set of parameters $(a,b)$ the map $f_{ab}$ is non-uniformly hyperbolic (see \cite{berhen} for a modern proof).

Also $C^r$-perturbations of Hénon maps appear naturally as a renormalized dynamics at every non-degenerate unfolding of homoclinic tangency of a dissipative hyperbolic periodic point, from  Palis-Takens Theorem (\cite{PT93}, \textsection III.4), that we will give in a revised form in Theorem \ref{PT2}. The perturbations of Hénon maps are in particular \emph{Hénon Like}. For this section, let us just mention that 
 Hénon like maps are dynamics which are close to the dynamics of a map $(x,y)\mapsto (P(x)+y,0)$, 
where $P$ is a unimodal map. The latter map has the same dynamics as $P$ restricted to the horizontal line. 

In this work we give some tools to understand the dynamics of Hénon maps, especially those with determinant small but reasonable enough to seem to include the map $H$ (from numerical evidence, see \textsection \ref{app3}). These tools are (pre)renormalization techniques and  normal forms for Hénon like maps.

The renormalization techniques generalize some renormalizations of unimodal maps.

In order to state several application and tools, let us recall and introduce a few definitions.

\begin{defi}
 A \emph{unimodal map} is a $C^2$-map $P$ of $\R$ which has a unique critical point at $0$ and whose second derivative $D^2P(0)$ is non zero, and so that $\infty$ is a super attracting fixed point of $P$:
\[P(x)\to \infty \quad\text{and}\quad D\left(\frac{1}{P(1/x)}\right)=\frac{x^2 DP(x)}{( P(x))^2}\to 0,\quad \text{as }x\to \infty.\]
The maximal invariant compact set of $P$ is either empty or either a segment bounded by a non-attracting fixed point and its preimage. 

The unimodal map is \emph{normalized} if $D^2P(0)=2$. It is always possible to normalize a unimodal map, via a linear conjugacy.
\end{defi}
 
For instance any quadratic map $x\mapsto x^2+a$ is unimodal and normalized.

\begin{defi} A \emph{Hénon $C^r$-like map} is any $C^r$-perturbation of a map of the form $(x,y)\mapsto (P(x)+y,0)$, where $P$ is a normalized unimodal map. 
\end{defi}
For instance any Hénon map $(x,y)\mapsto (x^2+a+y, -bx)$ is Hénon $C^r$-like (for every $r$ when $b$ small).

\begin{defi} 
A unimodal map $P$ is \emph{renormalizable} if there exists an open interval $J\subset I$  sent into itself by an iterate $P^N$ of the dynamics such that $P^N|J$ extends to a unimodal map of $\R$ with non-empty maximal invariant compact set included in $J$. The interval $J$ is a \emph{renormalization interval}.  
The \emph{renormalization} consists of conjugating $P^N|J$ to a normalized form. 
\end{defi}

This definition can be generalized to Hénon maps.
 \begin{defi} 
 A Hénon like map $f$ is \emph{renormalizable} if there exists an open interval subset $D\subset \R^2$  sent into itself by an iterate $f^N$ of the dynamics such that there exist
a Hénon like map $\mathcal R f$ and an embedding $\phi\in C^r(cl(D),\mathbb R^2)$ for which the following diagram commutes:
\begin{displaymath}
  \xymatrix{
    D\ar[d]_{\phi}\ar[r]^{ f^N} &          D\ar[d]_{\phi}\\
    \R^2 \ar[r]_{\mathcal Rf} &        \R^2 \\
  }
\end{displaymath}
and $\mathcal Rf$ has its (nonempty) maximal invariant compact set  contained in $\phi(cl(D))$. The integer $N$ is the \emph{period of the renormalization}, $D$ is the \emph{renormalization domain}. 
\end{defi}

For every renormalization interval $I_{0}$ of a quadratic map $P_{a_0}=x^2+a_0$, there exists a (maximal) parameter interval $\mathcal I$ such that for every $a\in \mathcal I$, the quadratic map $P_{a}$ has a renormalization interval $I_{a}$, depending continuously with $a$ and so that 
$I_{0}=I_{a_0}$.

For every renormalization parameter closed interval $\mathcal I$ of the quadratic family $P_a\colon x\mapsto x^2+a$, we can also associate a parameter domain of the Hénon family.

Indeed, from\footnote{Their definition of Hénon like maps is similar but not equal to the one here, in particular, they deal with analytic map. Our renormalization algorithm applies to such settings.} \cite{Ha11}, \cite{CLM}, there exists a connected parameter domain $\hat {\mathcal I}$ such that for every $(a,b)\in \hat {\mathcal I}$, the Hénon map $h_{ab}$ has a renormalization domain $D_r(a,b)$, depending continuously on $a$ and so that ${D_{r}}(a,0)\cap \R\times\{0\}=I\times\{0\}$. 

For every $b_0$ small, we can look a maximal parameter connected domain $\hat {\mathcal I}_{b_0}$ included in the strip $\R\times [-b_0,b_0]$.

\begin{ques}\label{quesstrip}
What is the geometry of $\hat {\mathcal I}_{b_0}$? Is there a unique maximal parameter connected domain $\hat {\mathcal I}_{b_0}$ containing $\mathcal I$? Does there is $b_0>0$ small
so that for all renormalization parameter intervals $\mathcal I$ of the
quadratic family, the sets $\hat {\mathcal I}_{b_0}$ are closed?
\end{ques}
Let $P_a(x)=x^2+a$. For $C>0$ and $\Lambda>1$, let $E_{C\lambda}$ be the subset of renormalization parameter intervals $\mathcal I$ such that for every $a\in \mathcal I$, if $I_a$ is the renormalization interval of $P_a$ and $p$ its period, it holds that the orbit of $I_a$ intersects $[-C,C]$ only at $I_a$ and:
\begin{equation}\tag{$\mathcal M^1$} \forall k\ge 0,\; \forall x\in \bigcap_{i=0}^{k-1} P^{-k}([-2,-C]\sqcup [C,2]), \quad  |D(P^k)(x)|>C \Lambda^k.\end{equation}

The techniques of renormalization of this work combined to those of \cite{Ha11} proves in \textsection \ref{app2} the following partial answer to Question \ref{quesstrip}.
\begin{theo}\label{partialans}
For every  $C>0$ and $\Lambda>1$, there exists $b_0>0$ so that for every $\mathcal I\in E_{C\lambda}$, the domain $\hat {\mathcal I}_{b_0}$ is a closed strip which stretches across $\R\times [-b_0,b_0]$.\end{theo}

From a result of Ma\~né \cite{Ma85}, for $r\ge 3$, Condition ($\mathcal M$) holds if there is no parabolic cycle (away of the renormalization interval) and if the orbit of the renormalization interval $(P_a^k(I_a))_{k=1}^{p-1}$ is distant to $0$. A natural next step would be to study the so-called \emph{parabolic renormalization} for Hénon-like map (for the study in the complex analytic setting see \cite{BSU12}).

The main novelties of these notes are the normal forms explained in \textsection \ref{sectapp2}. Together with the renormalizations (including linearizations) defined in \textsection \ref{sectapp3}, it is crucial to prove the following results:
\paragraph{Positive answer to a question of Lyubich \cite{zoology}}
  There is a parameter $(a,b)\in \R^2$ such that the Hénon maps $h_{ab}$ has two attracting cycles which attract Lebesgue almost every point which does not escape to infinity.
\paragraph{Proof of a numerical observation of Milnor \cite{zoology}, \cite{Mil92}} 
The parameter space of the composition $P$ of the two quadratic maps $x^2+c_1$ and $x^2+c_2$, with variable $(c_1,c_2)$ is similar to some regions of the parameter space of the Hénon family $(h_{ab})_{ab}$.
\paragraph{Lower bound to the Hausdorff dimension of Newhouse phenomena \cite{BedeSi}} Given a one dimensional  family of surface diffeomorphisms exhibiting a generic homoclinic unfolding, the Hausdorff dimension of the set of parameters for which the diffeomorphism has infinitely many sinks or sources is at least 1/2.

Also this work might be useful to work on the aforementioned Hénon conjecture.

\paragraph{Progress with respect to previous works on linearization}
The developed normal form techniques give sufficient conditions in order to prove that if $f\in C^r(\mathbb R^2,\mathbb R^2)$ and  if $(x_i)_{i}$ is an orbit in $\R^2$ of $f$, then there exist $C^{s}$-coordinates of a large neighborhood $W_i$ of each $x_i$ so that in these coordinates $f$ is linear. The assumption that we will use is hyperbolicity.

 If $(x_i)_{i}$ is a single fixed point $x$, sufficient conditions were first given by Poincaré in the analytic case, and then Sternberg in the smooth case. By denoting by $\lambda$ and $\sigma$ the eigenvalues of $D_xf$, the \emph{Sternberg condition of order $m$} is
 \begin{equation}\tag{$\mathcal S_m$}\lambda \not=  \lambda^i \sigma^j, \quad \text{and}\quad \sigma \not= \lambda^i\sigma^j,\quad \forall i,j\in \N,\; 2\le i+j\le m.\end{equation}
Sternberg \cite{Sternberg57, Sternberg58} showed that there is a function $V(s,r,\lambda, \sigma)$ such that if Sternberg condition is satisfied up to order $m$, where $m\ge  V$, then there is a $C^s$-linearization of $f$ at $x$, that is a diffeomorphism $\phi$ of a neighborhood of $x$ onto a neighborhood of $0$ so that $f \circ \phi = \phi\circ D_xf$.  

To quote Sell \cite{Sell85}, ``While there are several alternate proofs of Sternberg's Theorem [cf. \cite{Chen63}, \cite{Hartman64}, \cite{Nelson69}, \cite{Pugh70} and \cite{Takens71}], the implicit formulae of $V$ are very complicated. See \cite{Hartman64} (p. 257), for example.''

In the two dimensional case Sell \cite[\textsection3 p.1038]{Sell85} managed to prove the conditions $m\ge 2s$ or $m\ge 2s+1$, and $r\ge 3m$. In particular $s\le r/6$.

In another work, he showed that the estimate $m=2s$ is optimal for a precise choice of eigenvalues \cite{Sellceontrexemple}. 
 
A simple consequence of the presently presented normal form techniques is that we can obtain much better lower bound on $s$ with respect to $r$ and $m$ if we replace the Sternberg condition of order $m$ by the following:
\begin{equation} \tag{$\mathcal D_r$} \lambda \sigma^{r-2}<1\end{equation}
Condition $(\mathcal D_r)$ is the connected component of $(\R\setminus [-1,1])\times \{0\}$ of $\{(\sigma,\lambda)\in\R^2\; \text{satisfying }\;(\mathcal S_{r-1})\}$. In particular, Condition $(\mathcal D_r)$ implies $(\mathcal S_{r-1})$. 
A simple consequence of our renormalization techniques is that it implies $s=r-2$-linearization:
\begin{theo} If $r\ge 4$, if $f$ is a $C^r$ diffeomorphism of $\R^2$ for which $0$ is a hyperbolic fixed point, the eigenvalues of which satisfy $(\mathcal D_r)$ then it is $C^{r-2}$-linearizable. 
\end{theo}
\begin{proof}[Sketch of proof]
In Proposition \ref{foliation} we show that Condition $(\mathcal D_r)$ implies the existence of a pseudo-invariant $C^{r-1}$-foliation at the neighborhood of $0$. In a $C^{r-1}$-foliation chart the map $f$ has the form $f(x,y)=(g(x),h(x,y))$. Then by using Sternberg Theorem \cite{Sternberg57}, we can linearize without lost of derivatives the map $g$ to get the form $f(x,y)=(\sigma x,h(x,y))$ in a chart of class $C^{r-1}$. Actually the map $y\mapsto h(x,y)$ is of class $C^{r}$ for every $x$ and contracting. In Proposition \ref{linearization}, we first manage to get $\partial_y h(x,0)=\lambda$ for every $x$, and then we use a parameter dependent version of Sternberg Theorem to show linearize $y\mapsto  h(x,y)$. This cost one more derivative  to have a nice dependence with respect to $x$.
\end{proof}

Hence for this component of $(\mathcal S_{r-1})$, {\bf we get much a better function $V(r,s)=\min(r-1,s+1)$, instead of $V(r,s)=\min(r/3,2s)$ as previously}.    
it enables us to extend uniformly many hyperbolic results of unimodal map.

The developed normal formal techniques apply not only on fixed points but also along chain $(f^i9x))_{i=1}^n$ in a compact set (not necessarily invariant), the conjugacy being of $C^{r-2}$-bounded norm along arbitrarily long chains.
 
Moreover we will show the smoothness with respect to the parameter dependence. In the particular case of Hénon like maps, we will show that this parameter dependence is bounded ($C^{r-2}$-norm and size of the domain) {\bf even when the determinant goes to through zero}. This will be fundamental to show the aforementioned applications, since .


\subsection*{Notations} 
Let $M,M'$ be manifolds. A \emph{Riemannian metric} $g$ on $M$ is a continuous family $(g(z))_{z\in M}$ of inner products $g(z)$ of the tangent space $T_zM$ of $M$ at $z$. Let $g'$ be a Riemannian metric on $M'$ and $r\ge 1$.      
A map $f$ from $M$ into $M'$ is of class $C^r$ if its $r$-first derivatives, denoted by  $(D^i f(z))_{i\le r}$ depends continuously on $z\in M$. The map $D^i f(z)$ is linear from $T_zM\otimes \cdots \otimes T_zM$ into $T_{f(z)}M'$. The inner product $g(z)$ on $T_zM$ induces a norm $\|\cdot \|_z$ on $T_zM$. Canonically, the norm of  $u=(u_1,\dots,u_i)\in T_zM\otimes \cdots \otimes T_zM$ is $\|u \|:=\prod_k \|u_k \|_z$. The norm of  $D^i f(z)$ associated to $g$ and $g'$ is 
\[\|D^i f(z)\|_{g,g'}:=\max_{u\in T_zM\otimes \cdots \otimes T_zM, \|u\|_{g(z)}\le 1} \|D^i f(z)(u)\|_{g'(f(z))}.\]
If the norms $(\|D^i f(z)\|_g)_{z\in \R^2, 0\le i\le r}$ are bounded, we put:
\[\|f\|_{C^r,g,g'}:= \max_{z\in M, 0\le i\le r} \|D^i f(z)\|_{g,g'}\]
The space of $C^r$-maps of $\R^2$ endowed with  $\|\cdot \|_{C^r,g,g'}$ is a \emph{Banach space}, that is a complete vector space.

By $C^r$-\emph{endomorphism} we mean a map $f$ which is of class $C^r$, but which is not necessarily bijective, and whose differential is as well not necessarily surjective. 

If $\mathbb M$ is a manifold, a family $(f_p)_{p\in \mathbb M}$  of maps $f_p \colon M\to M'$ is\emph{ of class} $C^r$ if the following map is of class $C^r$: $M\times \mathbb M\ni(x,p)\mapsto f_p(x)\in M'$.

We recall that a \emph{cone field} $\chi$ on a domain $N$ of $\mathbb R^2$ is the data for every $z\in D$ of a nonempty open interval $\chi(z)$ of $\mathbb P^1\mathbb R$.

\section{Palis-Takens renormalization revisited}\label{PTrev}

Let us give more Hénon like related  definitions. 

\begin{defi} For $r\ge 1$, a \emph{Hénon like} family of class $C^r$ is a family of diffeomorphisms $(f_{a})_a$ of $\R^2$ which can be written in the form: 
\[f_{a}\colon (x,y)\mapsto (x^2+a+y,-b x) +(A_{a},b\cdot B_{a})(x,y),\]
 where $(x,y,a)\mapsto (A_a(x,y), B_{a}(x,y))$ is of class $C^r$.
 
The Hénon like family $(f_a)_a$ belongs to the actual Hénon family when $A=B=0$.

If the $C^r$-norms of $(x,y)\mapsto A_a(x,y)$ and $(x,y)\mapsto B_{a}(x,y)$ are smaller than $\delta$, the map $f_a$ is said \emph{Hénon $C^r$-$\delta$-like}.  If moreover $(x,y,a)\mapsto A_{a}(x,y)\in C^2(\R^3,\R^2)$ and $(x,y,a)\mapsto B_{a}(x,y)\in C^2(\R^3,\R^2)$ are $\delta$-small for every $a$, the family of maps $(f_a)_a$ is \emph{Hénon $C^r$-$\delta$-like}. 
\end{defi}
\begin{rema} The determinant of the Hénon like map is in $[(1-3\delta)b,(1+3\delta)b]$. Hence the determinant is small whenever $b$ is small.
\end{rema}

\begin{rema} We observe however that  the above expression of Hénon $\delta$-$C^r$-like map or family is much more precise that the one given in \cite[Thm. 1 P.47]{PT93}, \cite[App. A.2]{YW} or  \cite[Thm. 2.1]{Mora-viana}.  Such a form is crucial to study the possible number of attracting cycles in the Hénon family \cite{zoology}. \end{rema}

Let $(f_\mu)_{\mu\in \mathbb R}$ be a smooth family of $C^{\infty}$-maps of a surface having a hyperbolic fixed point $\Omega_\mu$.

We suppose that the eigenvalues $(\lambda_\mu,\sigma_\mu)$ of $\Omega_\mu$ satisfies at $\mu=0$ the following condition:
\begin{equation*}
\tag{$\mathcal R$}
 0<\lambda_0<1< \sigma_0 ,\quad   \lambda_0\cdot \sigma_0<1\quad \text{and}\quad \lambda_0^i\cdot \sigma_0^j\not= 1, \quad \forall i,j\not=0 \end{equation*}
The same condition is asked in \cite{PT93}. For every $r\ge 1$, this implies the existence of a $C^r$-chart $\phi_\mu$ \emph{linearizing $f_\mu$} at a neighborhood $D$ of $\Omega_\mu$, for every $\mu$ small \cite{Ta71}. 
This means that in $\phi_\mu$-coordinates $f_\mu|D\cap f_\mu^{-1}(D)$ has the form:
\[\phi_\mu^{-1} \circ f_\mu \circ \phi_\mu = : L_\mu\colon (x,y)\mapsto (\sigma_\mu x,\lambda_\mu y)\]
Moreover the family $(\phi_\mu)_\mu$ is $C^r$ ({\it i.e.} the map $(x,y,\mu)\mapsto \phi_\mu(x,y)$ is of class $C^r$). 

In these charts,  the point $\Omega_\mu$ is mapped to $0$, local stable and unstable manifolds of $\Omega_\mu$ onto $\{0\}\times \R$ and $ \R\times \{0\}$ respectively. 

We suppose that the homoclinic tangency holds at $\mu=0$ and that a \emph{non-degenerate unfolding holds} with $(f_\mu)_\mu$.\index{non-degenerate unfolding of a homoclinic tangency} This means that an iterate $f^N_0$ sends a point $P=\phi_\mu(p,0)$ to a point $Q=\phi_\mu(0,q)$, 
and for every $\mu$ small, and on a neighborhood $D_P$ of $P$ sent by $f_\mu^N$ into $D$,  $f_\mu^N|D_P$ has the following form (in the coordinates given by $\phi_\mu$):
\[ (p+x, y)\in D_P\mapsto  (\xi x^2+\mu +\gamma \cdot y, q+\zeta \cdot x)+ E_\mu(p+x, y)\in D\]       
where $\xi$, $\gamma$ and $\zeta$ are non-zero constants (independent of $\mu$), and $E_\mu=(E_\mu^1,E_\mu^2)\in C^{r}(\R\times \R^{2}, \R^2)$ satisfies:
\begin{enumerate}[$(\mathcal E_1)$]
\item $(x,y,\mu)\mapsto E_\mu(x,y)\in \R^2$ is $C^r$ and  the first coordinate $E_\mu^1$ of $E$ satisfies $\partial_\mu E_\mu^{1}(P)=0$.
\item $E_0(P)=0$.
\item $\partial_x E_0^{1}(P)=\partial_y E_0^{1}(P)=\partial_{xx} E_0^{1}(P)=0$, $\partial_x E_0^2(P)=0$. 
\end{enumerate}

\begin{figure}[h!]
    \centering
        \includegraphics{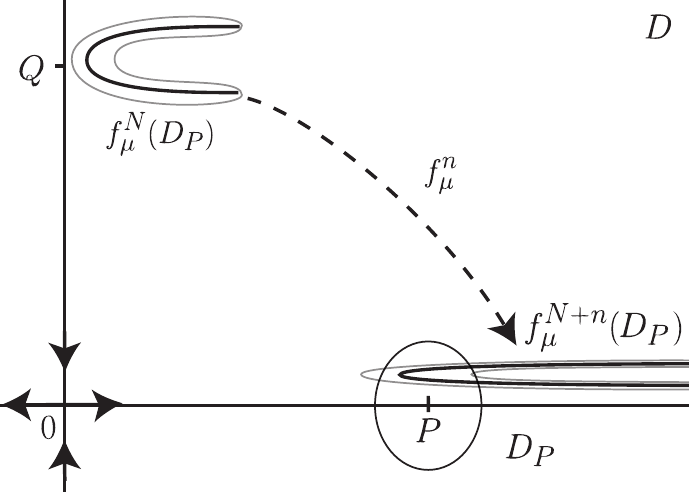}
    \caption{Unfolding of a Homoclinic tangency.}
\end{figure}
Here is an improvement of the celebrated Palis-Takens Theorem \cite{PT93} proved in \textsection\ref{Proof of Palis-Takens }.
 \begin{theo}\label{PT2} Let $r\ge 2$ and $\delta>0$. 
 For a one parameter family $(f_\mu)_\mu$ as above, there exists for each positive integer $n$, reparametrization $\mu=  M_n(\tilde\mu)$ of the $\mu$ variable and $\tilde \mu$-dependent coordinate transformations $(\tilde x,\tilde y)\mapsto (x,y)= \Psi_{n,\tilde \mu} (\tilde x, \tilde y)$ such that
 \begin{itemize}
 \item for each compact set $K$ in the $(\tilde \mu, \tilde x,\tilde y)$-space, the images of $K$ under the maps 
  \[(\tilde x,\tilde y, \tilde \mu)\mapsto (x,y,\mu)= ( \Psi_{n,\tilde \mu} (\tilde x, \tilde y),\tilde M_n(\tilde \mu))\]
 converge, for $n\to \infty$, in the $(x,y, \mu)$ space to $(P,0)$.
 \item the domains of the maps 
   \[(\tilde x,\tilde y, \tilde \mu)\mapsto \mathcal R f_{\tilde \mu}(\tilde x, \tilde y):=\Psi_{n,\tilde \mu}^{-1}\circ ((f|D)^n\circ (f^N|D_P))_{M_n(\tilde \mu)} \circ \Psi_{n,\tilde \mu}(\tilde x,\tilde y)\]
   converge, for $n\to \infty$, to all $\R^3$.
 \item Moreover for every compact set $K$ in the $(\tilde x,\tilde y, \tilde \mu)$ when $n$ is large, the restriction to $K$ of the family $(\mathcal R f_{\tilde \mu})_{\tilde \mu}$ is Hénon $C^r$-$\delta$-like (and with small determinant).  
\end{itemize}
\end{theo}
The improvement is on the precision of the form the renormalized map $\mathcal R f_{a_\mu}$, in Palis-Takens merely said that this family is $C^r$-close to the family of maps $((x,y)\mapsto (x^2+a+y,0))_a$. 

\begin{rema}\label{preNH} We have 
  $M_n(\tilde \mu)= {\sigma_\mu^{-2n} \tilde \mu -\lambda_\mu^n\gamma q+\sigma_\mu^{-n}p}$.

Also there exists an attractive cycle of period $n+N$ for all $n$ large and $\tilde \mu\in [-0.1,0.1]$, that is for 
$\mu\in M_n([-0.1,0.1])$ which is an interval of length about $\sigma^{-2n}_\mu$ and about $\sigma^{-n}_\mu$-distant to $0$, with uniform bounds in $n$.
\end{rema}

Theorem \ref{Misurenorm} of this manuscript will generalize the above Palis-Taken Theorem revisited for domain $D$ in which are in a neighborhood of a hyperbolic set, in a dissipative context. This will be useful together with the linearizing normal form.

\section{Fibered normal forms and linearization for Hénon like maps}\label{sectapp2}
In this section we defined an equivalent to the Birkhoff normal form for Hénon Like maps. The terminology is similar but different to the generalization given in the appendix of \cite{Guguyoyo10}. Actually by loosing two derivatives and assuming some dissipation, we reach a linear form.

Let $r\ge 1$ and let $f$ be a $C^r$ map of $\mathbb R^2$ {\bf not necessarily (locally) invertible}.  

\subsection{Definitions}
\paragraph{Fibered normal forms}
\begin{defi}
A \emph{$(C^r)$-semi-local chart} is the data of a pair $(y_\sa,\chi_\sa)$ where $y_\sa$ is  a $C^r$-diffeomorphism from the product of two $\R$-segments $I_\sa^u\times I_\sa^s$ containing $0$ in its interior, onto its image  $Y_\sa\subset \mathbb R^2$ and $\chi_\sa$ is a cone field on $Y_\sa$ such that
$D_zy_\sa(1,0)$ is in $\chi$ for every $z\in I_\sa^u\times I_\sa^s$.

We also define $\partial^s Y_\sa:= y_\sa((\partial I_\sa^u)\times I_\sa^s)$ and $\partial^u Y_\sa:= y_\sa( I_\sa^s\times(\partial I_\sa^u))$.\end{defi}

We recall that a (non-necessarily connected) graph $(\sY, \sT)$ with vertexes $\sY$ and arrows $\sT\subset \sY^2$ is \emph{convergent} if to every vertex, it ends at most one arrow.  It is \emph{divergent} if from every vertex, it starts at most one arrow. A chain of symbols $\underline \sa:= (\sa_i)_{i=1}^N\in \sY^N$ is \emph{admissible} if $\st_i=(\sa_i,\sa_{i+1})\in \sT$ for every $i<N$. 

Let $f$ be in  $C^r(W,W)$, with $W$ an open set of $\R^2$. 
\begin{defi}A \emph{$(C^r)$-system of charts} adapted to $f$ is a triplet $(\sY, \sT, \mathcal Y)$, where $(\sY, \sT)$ is a divergent and convergent graph and $\mathcal Y$ is a family of $(C^r)$-semi local chart $\mathcal Y =:(y_\sa,\chi_\sa)_{\sa\in \sY}$ such that for every 
$\st=(\sa,\sb)\in \sT$ it holds: 
\begin{itemize}
\item[$(S_0)$] $f(y_\sa(0))=y_\sb(0)$,
\item[$(S_1)$] $y_\sb(I_\sb^u\times \{0\}) $ is a segment of $f(y_\sa(I_\sa^u\times \{0\}))$,
\item[$(S_2)$] for every $z\in Y_\sa\cap f^{-1}(Y_\sb)$, $D_zf(\chi_\sa(z))\subset \chi_\sb(f(z))$,
\item[$(S_3)$] for every $y\in I_\sa^s$, 
 the curve $\{f\circ y_\sa(x,y):\; x\in I_\sa^u\}$ intersects $Y_\sb$ at a single segment with end points in the interior of $\partial^s Y_\sb$. 
\end{itemize}
\end{defi}

Two systems of charts $(\sY, \sT, \mathcal Y)$ and $(\sY', \sT', \mathcal Y')$ are {\bf equivalent} if they are associated to the same graph $(\sY, \sT)=(\sY', \sT')$ and for every $\a\in \sT$ the maps $y_\a$ and $y_\a'$ have the same image $Y_a$.

We look for systems of charts in which the representation of $f$ is the simplest possible, that is the family $(F_{\st})_{\st=(\sa,\sb)\in \sT}$ defined by:
\[F_{\st} = y_{\sb}^{-1}\circ f\circ y_{\sa}.\]
\begin{defi} Let $(\sY,\sT,\mathcal)$ be a system of charts adapted to $f $.

\begin{itemize}
\item It is of \emph{Type $A$} if for all $\st \in \sT$, there exists $\sigma_\st\in \R$ such that  $F_{\st}(x,0)=\sigma_\st\cdot x$ for every $x$.
\item It is of \emph{Type $B$} if for all $\st \in \sT$, the map $F_{\st}$ is of the form:
$F_{\st}(x,y)= (\sigma_\st\cdot x,g_\st(x,y))$, for all $x,y$.
\item It is of \emph{Type Linear} if  for all $\st \in \sT$, the map $F_{\st}$ is linear: there exist $\sigma_\st,\lambda_\st\in \R$ such that 
$F_{\st}(x,y)= (\sigma_\st\cdot x,\lambda_\st\cdot y)$, for all $x,y$.
\end{itemize}
 \end{defi}
We notice that a linear system of charts of type $B$ and so of type $A$. A nice hypothesis to make a system of chart of better type is \emph{hyperbolicity}:
\begin{defi} 
A system of charts is \emph{$(C,\Lambda)$-hyperbolic} for  $C>0$, $\Lambda>1$, if for every admissible sequence $(\sa_1, \dots, \sa_k)\in \sA^k$, we have
\begin{equation}\tag{$\mathcal {M}$}
\forall i\le k, \;\forall j\le k-i, \; \forall u\in \chi_{\sa_i},\quad \|Df^j|u\|\ge C\Lambda^j 
\end{equation}
\end{defi}
We remark that condition $(\mathcal M)$ is the canonical generalization for surface maps of condition $(\mathcal M^1)$ for unimodal maps.

\paragraph{Fibered and parameter dependent normal forms}

Given a submanifold $\M\subset \R^d$, $d\ge 1$, a family $(f_\mu)_{\mu\in \M}$ of maps of $\R^2$  is of class $C^r$ if the following is of class $C^r$ 
\[(z,\mu)\in \mathbb R^2\times \M\to f_\mu(z).\]
\begin{defi} 
A \emph{$(C^r)$-parameter dependent system of charts} is the data of a family  of systems of charts $(\sY(\mu), \sT(\mu), \mathcal Y(\mu))$ of $f_\mu$ for every ${\mu\in \M}$ so that the graphs do not depend on $\mu$: $\sY(\mu)=\sY$ and $\sT(\mu)=\sT$, and  such that the domain of $I_{\sa\mu}^u\times I_{\sa\mu}^s$ of $y_{\sa\mu}\in \mathcal Y(\mu)$ and  the cone  $\chi_{\sa\mu}$ depend continuously on $\mu\in \M$, and so that the following map is $C^r$:  
 \[\cup_{\M }  I_{\sa\mu}^u \times I_{\sa\mu}^s\times \{\mu\}\ni (z,\mu)\rightarrow y_{\sa\mu}(z)\in \mathbb R^2\]
\end{defi}

To get linear system of charts, we are going to lose derivatives at several steps. In order to lose the less possible, let us define intermediate classes of regularity.

\begin{defi} 
A \emph{$(C^r_A)$-parameter dependent system of charts} (resp. \emph{$(C^r_B)$}) is the data of a family  of systems of charts $(\sY(\mu), \sT(\mu), \mathcal Y(\mu))$ of $f_\mu$ for every ${\mu\in \M}$ so that the graphs do not depend on $\mu$: $\sY(\mu)=\sY$ and $\sT(\mu)=\sT$, and  such that the domain of $I_{\sa\mu}^u\times I_{\sa\mu}^s$ of $y_{\sa\mu}\in \mathcal Y(\mu)$ and  the cone  $\chi_{\sa\mu}$ depend continuously on $\mu\in \M$, and so that the following map:  
 \[\cup_{\M }  I_{\sa\mu}^u \times I_{\sa\mu}^s\times\{\mu\}\ni (z,\mu)\rightarrow y_{\sa\mu}(z)\in \mathbb R^2\]
has its partial derivatives of the form $\partial_x^i\partial_y^j\partial_\mu^k$ well defined and continuous for every $(i,j,k)\in \Delta^r_A$ (resp. $\Delta^r_B$), with:
\[\Delta^r_A:= \{(i,j,k):\; i+j+k\le r,\; k\le r-2\}\]
\[\Delta^r_B:= \Delta^r_A\setminus \{r\}\times \N^2\]
\end{defi}

\begin{defi} For $r\ge 2$, the norm of classes $C^r$, $C^r_A$ and $C^r_B$ of a parametrized system of charts $(\sY,\sT)$ are respectively:
\[\sup_{\sa\in \sY} ( \|(z,\mu)\mapsto y_{\sa\mu}(z) \|_{C^r} , \|(z,\mu)\mapsto y_{\sa\mu}^{-1}(z)\|_{C^1}).\]
\[\sup_{\sa\in \sY} (	\|(z,\mu)\mapsto y_{\sa\mu}^{-1}(z)\|_{C^1}, \sup_{(z,\mu)}\{\|\partial_x^i\partial_y^j\partial_\mu^k y_{\sa\mu}(z) \|:\;  (i,j,k)\in \Delta^r_A\}).\]
\[\sup_{\sa\in \sY} (	\|(z,\mu)\mapsto y_{\sa\mu}^{-1}(z)\|_{C^1}, \sup_{(z,\mu)}\{\|\partial_x^i\partial_y^j\partial_\mu^k y_{\sa\mu}(z) \|:\;  (i,j,k)\in \Delta^r_B\}).\]
\end{defi}

Actually it is very easy to make a system of class $A$ from the following:
\begin{prop}\label{semilinearization}
Let $r\ge 2$ and let $(f_\mu)_\mu$ be a family $C^{r}$-map of $\R^2$, preserving a $(C,\Lambda)$-hyperbolic  $C^r_A$-system of charts $(\sY, \sT, \mathcal Y(\mu))_\mu$ with bounded $C^r$-norm. 
Then  for every $\mu$, there exists  an equivalent $C^r_A$-system of charts $(\sY, \sT, \mathcal Y'(\mu))_\mu$ of type $A$ adapted to  $(f_\mu)_\mu$. Moreover its norm is bounded in function of $(C,\lambda)$, and the norms of $(\sY, \sT, \mathcal Y(\mu))_\mu$ and $f$.
\end{prop} 
The proof of this Proposition is somehow classical and will be done in \textsection \ref{a2B}. 

Let us give a nice condition to construct linear system of charts. Let $f\in  C^{r}(W, W)$, with  $V\subset W\subset \mathbb R^2$ open subset and $r\ge 2$. We will suppose that there exists a cone field $\chi $ on  $W$ \emph{pseudo-invariant relatively to $V$} (this means that 
$cl(D_zf(\chi(z)))\subset  \chi(f(z))$, for all $z\in V$) and that there exist $C>0$ and $\Lambda>1$ such that for every $k\ge 1$:

\begin{equation}\tag{$\mathcal D_r$} 
\left\{\begin{array}{c}
\forall z\in V,\; D_zf(\chi(z)) \text{ is $C$-distant to } \P\R^1\setminus \chi(f(z)),\\ \forall z\in \bigcap_{i=0}^{k-1} f^{-i}(V),\; u\in \chi(z),
\quad\|D_z f^k|u\|^{r-3} \cdot |det \, D_z f^k |\le (C\Lambda ^k)^{-1}.\end{array}\right.\end{equation}


\begin{rema}\label{rema2.9} The first line of $(\mathcal D_r)$ implies the existence of a dominated splitting on $W$; thus there exists $C'>0$ and $\Lambda'>1$ (depending only on $C$) so that: 
\[\forall z\in \bigcap_{i=0}^{k-1} f^{-i}(V),\; \forall u\in \chi(z),\quad \frac{ |det \, D_z f^k |}{\|D_z f^k|u\|^2}\le (C'\Lambda'^k)^{-1}.\]
Hence, by taking $C'<C$ and $\Lambda'<\Lambda$, for all $p>0$, $q>0$, it holds by the second line of $(\mathcal D_r)$:
\[\forall z\in \bigcap_{i=0}^{k-1} f^{-i}(V),\; u\in \chi(z),\quad 
 \|D_z f^k|u\|^{(r-3)q-2p} \cdot |det \, D_z f^k |^{p+q}\le (C'\Lambda'^k)^{-p-q}\]
For suitable values of $p,$ so that $p+q=1$, 
the number $(r-3)q-2p$  reaches all the values in $\{-2,\dots, r-3\}$. Thus it comes that $(\mathcal D_r)$ implies:
\begin{equation}\tag{$\mathcal D_r$} \forall z\in \bigcap_{i=0}^{k-1} f^{-i}(V),\; u\in \chi(z),
\quad\|D_z f^k|u\|^{i} \cdot |det \, D_z f^k |\le (C'\Lambda'^k)^{-1},\quad \forall i\in\{-2,\dots,  r-3\}\end{equation}
\end{rema} 

We will see how to construct a system of charts of type $B$ thanks to the following.
\begin{prop}\label{foliation}  Let  $V\subset W\subset \mathbb R^2$ be open. Let $r\ge 1$ and let  $\mathbb M\subset C^{r}(W, W)$ be a submanifold. Suppose that there exists a pseudo-invariant cone field $\chi $ on  $W$ for  every $f\in \mathbb M$ so that Condition $(\mathcal D_r)$ holds for $C>0$, $\Lambda>1$ .  Then there exists a line field $e(f)\colon W\to \mathbb P^1$ depending only on  $C$, $\Lambda$ and $\|Df\|_{C^{r}}$, such that:
\begin{itemize} 
\item  For every $z\in V$, $T_z f(e (f)(z))$ is parallel to $e(f)(f(z))$, and $e(f)$ is in the interior of the complement of $\chi$.
\item  The map $(z, f) \mapsto e(f)(z) $ is of class $C^{r-1}$.  
\end{itemize}
\end{prop}
 This proposition will be proved in \textsection \ref{Pseudo invariant stable foliation}. It is also similar to a result 
 in \cite{PSW97}, in the context of partially hyperbolic diffeomorphism (that is $V=W$). This proposition will be done by integrating the vector field $e(f)$.

To linearize a system of type $B$, we face to another difficulty: the second coordinate $g_\st$ of $F_\st(x,y)= (\sigma_\st, g_\st(x,y))$  is possibly not always of maximal rank, which makes it impossible to linearize, that is to reach the form $g_\st(x,y)=\lambda_\st\cdot y$, with $\lambda_\st$ a real number.

Indeed, up to now, we did not assume $f$ invertible.  This is why we will ask $f$ to be Hénon $\delta$-$C^1$-like with $\delta<1/2$.

\begin{prop}\label{linearization}
Let $r\ge 2$ and let $(f_\mu)_\mu$ be a family $C^{r}$-map of $\R^2$, preserving a $(C,\Lambda)$-hyperbolic $C^{r}_B$-system of charts $(\mathcal Y(\mu),\sY, \sT)_\mu$ of type $B$ with bounded norm. We suppose that $f_\mu$ is Hénon $C^1$-$\delta$-like with $\delta<1$. Suppose that condition $(\mathcal D_r)$ holds among admissible chains.
Then $(f_\mu)_\mu$ preserves an equivalent $C^{r-2}$-system of charts $(\mathcal Y'(\mu),\sY, \sT)_\mu$ which is linear.
Moreover the $C^{r-2}$-norm of $(\mathcal Y'(\mu),\sY, \sT)_\mu$ is bounded in function of $C,\Lambda,\delta$, the $C^r_B$-norm of $(\mathcal Y(\mu),\sY, \sT)_\mu$ and the $C^r$-norm of $(f_\mu)_\mu$.
\end{prop}
\begin{rema}To say that \emph{$(\mathcal D_r)$ holds uniformly among admissible chains} means that there exists $C>0$, $\Lambda>1$ such that for every admissible chain $(\sa_l)_{l=1}^j$, for every $u\in \chi_{\sa_1}$, for all $k\le j$ it holds 
\begin{equation}\tag{$\mathcal D_r$} 
\left\{\begin{array}{c}
\forall z\in Y_{\sa_{1}}\cap f^{-1}(Y_{\sa_{2}}),\; D_zf(\chi_{\sa_{1}}(z)) \text{ is $C$-distant to } \P\R^1\setminus \chi_{\sa_{2}}(f(z)),\\ 
\forall z\in \bigcap_{i=k}^{j} f^{-i+k}(Y_{\sa_{i}}),\; u\in \chi_{\sa_{k}}(z),
\quad\|D_z f^k|u\|^{r-3} \cdot |det \, D_z f^k |\le (C\Lambda ^k)^{-1}.\end{array}\right.\end{equation}
\end{rema} 

The proof of Proposition \ref{linearization} will be done in \textsection \ref{proofof{estimaterenor}}. 

\subsection{Construction of systems of charts}
The aim of the next section is to find $(C^r)$-parameter dependent system of charts with bounded norm while the graph $(\sY , \sT)$ can be unbounded, {\it ie} it can contain admissible chains of arbitrarily length.  Throughout this section $r\ge 4$ is an integer.
\subsubsection{System of charts for hyperbolic sets}
\label{PSMH}
Let $K$ be a \emph{hyperbolic set} of $\mathbb R^2$ for a $C^{r}$-diffeomorphism\footnote{With more care, a similar construction can be done also for endomorphisms.} $f$, with $r\ge 1$. We recall that this means that $K$ is invariant ({\it i.e.} $f(K)=K$), and that there exists a splitting of $T\R^2|K$ into two $Df$-invariant directions $E^s$ and $E^u$ which are respectively contracted and expanded.

For $k\in K$ and $\epsilon>0$, we denote the $\epsilon$-local stable and unstable manifolds by respectively:
\[W^s_\epsilon(k):= \{z\in \R^2:\; \epsilon > d(f^n(z), f^n(k))\to 0 \text{ as } 0\le n\to \infty\}\]
\[W^u_\epsilon(k):= \{z\in \R^2:\; \epsilon > d(f^{-n}(z), f^{-n}(k))\to 0\text{ as } 0\le n\to \infty\}.\]  
They are $C^{r}$-submanifolds. 
\paragraph{System of type A}
For $\delta>0$ small, as the manifolds $W^s_\delta (k)$ depend continuous on $k\in K$, there exists a  continuous family of $C^r$-foliations $(\mathcal F(k))_{k\in K}$ on a neighborhood such that $W^s_\delta(k)$ is a plaque of $\mathcal F(k)$ for every $k\in K$. For instance this can be obtained by considering the foliation $\mathcal F(k)$ whose leaves are the translated of the plaque $W^s_\delta(k)$ in the direction $E^u(k)$.

For $\epsilon <\delta$ small, for every $k\in K=\sY$, let $(y_k(s,0))_{s\in I_k^u}$ be an arc length parametrization of $W^u_{\epsilon}(z)$. Note that the segment $I_k^u$ is close to $[-\epsilon, \epsilon]$. 

For every $x\in I^u_k$, let $(y_k(x,y))_{y\in [-\epsilon, \epsilon]}$ be the arc length parameterization of a $\mathcal F(f)$-plaque containing $y_k(x,0)$, so that $(\partial_x y_k(x,y),\partial_y y_k(x,y))$ is direct. Put $I_k^s:= [-\epsilon, \epsilon]$. Let $\chi$ be a cone field centered at a continuous extension of $E^u$ on a neighborhood of $K$ s.t. $\chi$ is sent into by $Df$.  
Put $\sY:= K$ and let $\sT$ be the graph of $f|K$:  $\sT:= \{(k,f(k))\}_{k\in K}$. 

The family $\mathcal Y(f):= (y_k)_{z\in K=\sY}$ forms with $(\sY,\sT)$ a $C^{r}$-system of charts.
By Proposition \ref{semilinearization} we can assume this system of $C^{r}$-type $A$ by changing the $x$-coordinates.

\paragraph{System of $C^{r}$-type B}
We assume that Proposition \ref{foliation} holds on a neighborhoods $V\subset W$ of $K$. Let $\mathcal F(f)$ be the foliation on $W$ whose leaves are tangent to $e(f)$. We remark that $W^s_\epsilon(k)$ is a plaque of $\mathcal F(f)$ for every $k\in K$. Hence we can take in the above construction the foliation $\mathcal F(k)$ equal to a restriction of $\mathcal F(f)$, in order to obtain a system of type $B$(by pseudo-invariance of the foliation $\mathcal F(f)$). However the foliation is now given by the integration of a vector field of class $C^{r-1}$, hence the obtained system is of class $C^{r}_B$. 

\paragraph{System of $C^{r-2}$-type Linear}
If the map $f$ is Hénon $C^1$-like, we can apply Proposition \ref{linearization} to assume the system of charts Linear and of class $C^{r-2}$ with bounded norm. We remark that we lost two derivatives to have this linearization.
%

\paragraph{Norm and parameter dependence} 
We are going to study the parameter dependence of the above systems of charts of type $A$ (resp. B, resp. Linear). 

For $C^{r}$-perturbation $f'$ of $f$, the hyperbolic set $K$ persists to a hyperbolic compact set $K(f')$ which is homeomorphic to $K$ via a homeomorphism $h(f')$ which conjugates the dynamics of $f'|K(f')$ and $f|K$ and so that (see \cite{Yoyointro}):
\begin{itemize}
\item $h(f)=id|K$.
\item for every $k\in K$, the map $f'\mapsto h(f')(k)$ is of class $C^{r}$.
\end{itemize}
This \emph{hyperbolic continuity} enables us to identify $K(f')$ to $\sY$ and $Graph(f'|K(f'))$ to $\sT$. 
Also, we can construct in the same way systems of charts $(\sY, \sT, \mathcal Y(f'))$ of type $A$ (resp. $B$, resp. linear). Indeed the construction of the sytem of charts depend smoothly on $f$ by hyperbolic continuity (and resp. Proposition \ref{foliation}, and resp. moreover Prop. \ref{linearization}). In other words, the parameter dependent systems
are of type $A$ is of $C^{r}_A$-norm bounded (resp. of type $B$ norm $C^{r}_B$-bounded, resp. Linear of $C^{r-2}$-norm bounded).
\begin{rema} The lost of derivatives does not seem to be avoidable: it already occurs in the Sell Theorem \cite{Sell85} in a more dramatic way as discussed in the introduction, where the hyperbolic compact set is a single fixed point.
\end{rema}

\subsubsection{System of charts of Hénon like maps induced by renormalization intervals of quadratic maps}\label{PSHL}
Let $C>0$, $\lambda>1$ and $P_a(x)=x^2+a$ be a quadratic map so that $a$ belongs to a renormalization parameter interval $\mathcal I\in E_{C\lambda}$, where $E_{C\lambda}$ was defined for Theorem \ref{partialans}. We recall that $I_a$ denotes the renormalization interval of $P_a$ containing $0$. Let $p$ be the period of $I_a$. We recall that the orbit of $I_a$ intersects $[-C,C]$ only at $I_a$ and
\begin{equation}\tag{$\mathcal M^1$} \forall k\ge 0,\; \forall x\in \cap_{i=0}^k P^{-i}([-2,-C]\sqcup [C,2]), \quad  |D(P^k)(x)|>C \Lambda^k.\end{equation}
Hence the compact subset $K_C=\cap_{n\ge 0} P_a^{-n}([-2, 2]\setminus (-C,C))$ is hyperbolic. We notice that $P^k(I_a)$ is included in a gap $\sa_k$ of $K_C$, for every $k\le p$. 
     
Let $\sY:= \{\sa_k,k\le p\}$ and let $\sT:= \{(\sa_k,\sa_{k+1}),\; 0\le k< p\}$.

By $(\mathcal M^1)$, for every $r\ge 1$, there exists $\epsilon>0$ and $\theta>0$, depending only on $C,\lambda$, so that if $\chi$ is the cone field centered at $(1,0)$ with angle $\theta$ on $\R^2$, it holds with $C'<C$, $\Lambda':= \sqrt{\Lambda}$ and $V_\epsilon$ the $10\epsilon$-neighborhood of $([-2,-C]\sqcup [C,2])\times \{0\}$  that for every $\epsilon$-$C^1$-perturbation $f_a$ of $h_{a0}\colon (x,y)\mapsto (x^2+a+y,0)$, we have for all $k\ge 0$ and $z\in \bigcap_{i=0}^{k-1} f_a^{-i}(V_\epsilon )$:
\begin{equation}\tag{$\mathcal M$} \forall u\in \chi(z),\;  \|D_zf^k_a|u\|\ge C'\Lambda'^k\|u\|\end{equation}
\begin{equation}\tag{$\mathcal D_r$} 
\left\{\begin{array}{c}
D_zf_a(\chi(z)) \text{ is $C'$-distant to } \P\R^1\setminus \chi(f_a(z)),\\
\forall u\in \chi(z),  \quad\|D_z f_a^k|u\|^{r-3} \cdot |det \, D_z f_a^k |\le (C'\Lambda'^k)^{-1}.\end{array}\right.\end{equation}

We apply to $f_a$, $V_\epsilon$ and $W_\epsilon$, Proposition \ref{foliation} to define a line field $e(f_a)$ on $W_\epsilon$, for every $f_a$ $\epsilon$-$C^{1}$-close to $h_{a0}$. Let  $\mathcal F(f)$  be the $C^{r-1}$-foliation integrating $e(f_a)$ on $W_\epsilon$. 

We recall that  $\sa_k$ is bounded by two hyperbolic points $\alpha_k,\beta_k$. Hence the points  $(\alpha_k,0)$ and $(\beta_k,0)$ are also hyperbolic for the map $(x,y)\mapsto (x^2+a+y, 0)$ and so they persist to hyperbolic points for $f_a$ (provided that $\epsilon$ is sufficiently small depending only on  $C,\lambda$).

Let $\gamma_{1}(f_a)$ be the segment of $\R\times\{0\}$  close to $\sa_1\times\{0\}$ bounded by the two 
leaves of $\mathcal F(f_a)$ containing the hyperbolic continuation of respectively $(\alpha_1,0)$ and $(\beta_1,0)$. For every $k\le p$, let $\gamma_{k}(f_a)$ be the image by $f_a^{k-1}$ of $\gamma_{1}^u(f_a)$. The curve $\gamma_{k}(f_a)$ has its tangent space in $\chi$. Its end points  are bounded by the leaves $\mathcal L(f_a)$ containing the hyperbolic continuation of $(\alpha_k,0)$ and $(\beta_k,0)$. 
Given $C$ and $\Lambda$, for $\epsilon = \epsilon(C,\Lambda)$ small, it holds for every $k$ that the curve $\gamma_{k}(f_a)$ is $C^{r}$-close to the horizontal in the following meaning: it is the graph of a $C^{r}$-small function $\rho_{k}\in C^{r}(J_k,\R)$ over a segment $J_k$ close to $\sa_k$. 

We now suppose $f_a$ of class $C^r$ and still $C^1$-$\epsilon$-close to $h_{a0}$.

We notice that for every $1\le k<p$, $f_a$ sends   $\gamma_{k}(f_a)$ onto $\gamma_{k+1}(f_a)$.  Let $x_k$ be the point in the interior of $J_k$ so that $f^{p-k}_a(\rho_{k}(x_k))$ is the intersection point of $\gamma_{p}(f_a)$ with $\{0\}\times \R$. 
We remark that  $I^u_{\sa_k}:= J_k-x_k$ contains $0$ in its interior. For $x\in I^u_{\sa_k}$, put $y_{\sa_k}(x,0):= (x_k+x,\rho_{k}(x_k+x))$. We remark that Conditions $(S_0)$ and $(S_1)$ are satisfied.

For every $x\in I^u_{\sa_k}$, let $(y_{\sa_k}(x,t))_{t\in [-\epsilon, \epsilon]}$ be the arc length parameterization of a $\mathcal F(f_a)$-plaque containing $y_{\sa_k}(x,0)$, so that $(\partial_s y_{\sa_k}(x,t),\partial_ty_{\sa}(x,t))$ is direct. Put $I_{\sa}^s:= [-\epsilon, \epsilon]$. 

As in the previous example, by changing the $x$-coordinate thanks to Proposition \ref{semilinearization}, the family $\mathcal Y(f_a):= (y_\sa)_{\sa\in \sY}$ forms with $(\sY,\sT)$ a system of $C^r$-charts of type $B$. It has a bounded $C^r_B$-norm (depending on $C$ and $\Lambda$) for $f$ $C^{r}$-close to $f_a$. 

As before, the construction depends smoothly on $f$. In other words, for every $C^{r}$-family $(f_{\mu})_{\mu \in \M}$ close to $f_a$, the parameter dependent system $(\sY, \sT, \mathcal Y(f_\mu))_{\mu \in \M}$ is of class $C^r_B$ with bounded norm. 

If moreover $f_a$ is a Hénon $C^{1}$-like family, we can change the system to an equivalent one of Linear type and class $C^{r-2}$, with bounded norm.

\section{Renormalization of Hénon like map}\label{algorithm of renormalization}\label{sectapp3}
\subsection{ Fibered Palis-Takens renormalization}
Let $(f_\mu)_{\mu \in \R}$ be a $C^{r}$-family of maps of $\R^2$. Let $(\sY, \sT, \mathcal Y(\mu))_\mu$ be a $C^{r}$-system of charts which is hyperbolic and Linear. 
  
Let $\sa_0,\sa_1\in \sY$ be so that on $D_P(\mu):= y_{\sa_0\mu}^{-1}(Y_{\sa_0\mu}\cap f_\mu ^{-1}(Y_{\sa_1\mu}))$ the map $y_{\sa_1\mu}^{-1}\circ f_\mu \circ y_{\sa_0\mu}$ has the form 
\[y_{\sa_1\mu}^{-1}\circ f_\mu \circ y_{\sa_0\mu} \colon (p+x, y)\in D_P \mapsto  (\xi x^2+\mu  +\gamma \cdot y, q+\zeta \cdot x)+ E_\mu (x, y)\in \R^2,\]    where $p\in int\, I^u_{\sa_0} $, $q\in int\, I^s_{\sa_1} $, $\xi\not=0$,  $\gamma\not=0$ and  $\zeta$ are constants (independent of $\mu$) and $E_\mu \in C^{r}(\R\times \R^2, \R^2)$  satisfies $(\mathcal E_1)$, $(\mathcal E_2)$ and $(\mathcal E_3)$ with $P:= (p,0)$. 

We recall that a chain of symbols $\underline \sa:= (\sa_i)_{i=1}^N\in \sY^N$ is \emph{admissible} if $\st_i=(\sa_i,\sa_{i+1})\in \sT$ for every $i<N$. Then $\lambda_{{\underline \sa}\mu}$ and $\sigma_{{\underline \sa}\mu}$ are defined by:
\[ F_{\underline{\sa}}:= F_{\st_{N-1} } \circ \cdots \circ F_{\st_1}=
y_{\sa_N\mu}^{-1}\circ  f_\mu \circ   y_{\sa_{N-1}\mu}\circ \cdots \circ y_{\sa_2\mu}^{-1}\circ  f_\mu \circ   y_{\sa_{1}\mu}
=:
(x,y)\mapsto ( \sigma_{{\underline \sa}\mu} x, \lambda_{{\underline \sa}\mu} y)\]

The following Theorem is then proven identically as Theorem \ref{PT2}.
\begin{theo}\label{Misurenorm} For every adapted chain of symbols $\underline \sa:= (\sa_i)_{i=1}^N\in \sY^N$ starting at $\sa_1$ and ending at $\sa_N= \sa_0$,  the map $f^{N}_\mu $ from $ W_{0,\mu}:= Y_{\sa_0\mu}\cap \bigcap_{i=1}^N f_\mu ^{-i}(Y_{\sa_i\mu})$ into $Y_{\sa_0\mu}$ is conjugated to the following map which is Hénon $C^r$-like: 
 \[\mathcal R f_{a_\mu}= \Psi_\mu\circ y_{\sa_0\mu}^{-1}\circ  f_\mu ^{N}\circ   y_{\sa_0\mu} \circ \Psi_\mu^{-1},
 \text{  where }\left\{\begin{array}{l} 
 \Psi_\mu(x,y):= \xi\cdot \sigma_{\underline \sa \mu}\cdot (x- p ,  \sigma_{\underline \sa \mu} {\gamma y}-\sigma_{\underline \sa \mu} \lambda_{\underline \sa \mu} q),\\
a_\mu:= \sigma_{\underline \sa \mu}^2 \cdot(\xi \mu +\lambda_{{\underline \sa}\mu}  \xi \gamma q-\frac{\xi p}{\sigma_{\underline \sa \mu}}).\end{array}\right. \]
Moreover, for every $\delta$, if $N$ is large enough in function of the norm and the hyperbolicity of $\mathcal Y$, the norm of $f$, $\gamma$, and $\xi$ then  the family $(\mathcal R f_{a_\mu })_{a_\mu }$ is Hénon $C^{r}$-$\delta$-like.\end{theo}

We saw in the previous section a way to get such $C^{r-2}$-systems of charts which are hyperbolic and linear, by assuming  $(f_\mu)_{\mu \in \R}$ of class $C^{r}$. Hence we lose in this process bounds on 2-derivatives. Actually we do not lose completely $2$-derivatives, since after the renormalization done, we can smooth the renormalization chart in order to recover the lost of these  two derivatives. However we lost the bounds.

We can reach new applications thanks to examples of \textsection \ref{PSMH} and \ref{PSHL}.
\subsection{Applications}

\subsubsection{A lemma to bound from below the Hausdorff dimension of Newhouse parameters} \label{app1}
The following proposition is crucial to show that for every surface diffeomorphism family $(f_a)_a$ for which a non-degenerate homoclinic unfolding of a hyperbolic fixed point $\Omega$ holds, with $|det\, D_{\Omega} f|<1$, then the set of parameters $a$ for which $f_a$ has infinitely many sinks (Newhouse phenomenum) has Hausdorff dimension at least $1/2$ \cite{BedeSi}. 

\begin{lemm}\label{prenh}
Let $(f_\mu)_\mu$ be a Hénon $C^2$-like family, of class $C^4$, such that a non-degenerate homoclinic unfolding holds with a periodic point $\Omega$  in a hyperbolic set $K$ at $\mu=\mu_0$. 

Then there exist $C>0$, $M\in \mathbb N$ and $N\ge 0$ such that for every periodic point $\Omega'$ of period $p'$ close to $\Omega$, a non-degenerate homoclinic unfolding holds with $\Omega'$ at $\mu_0'$ close to $\mu_0$. 

Moreover if $p'\ge M$, by denoting by $\sigma$ the unstable eigenvalue of $D_{\Omega'} f_{\mu_0'}^{p'}$, there exists a parameter interval $I_{\Omega'}$  such that:
\begin{itemize}
\item for every  $\mu\in I_{\Omega'}$, the map $f_\mu$ has an attracting cycle of minimal period $p'+N$,
\item the length of $I_{\Omega'}$ is in $[\sigma^{-2}/C,C \sigma^{-2}]$,
\item the distance between $I_{\Omega'}$  and  $\mu_0'$ is in $[\sigma^{-1}/C, C \sigma^{-1}]$.
\end{itemize}
\end{lemm}

\begin{proof} 
Let $(\sY, \sT, \mathcal Y(\mu))_\mu$ be the linear $C^2$-system of charts  associated to $K$ given by \textsection \ref{PSMH}.  We recall that it has a bounded norm and that $\sY=K$, $\sT= Graph(f_{\mu_0}|K)$.

For $\epsilon$ small enough, $W^s_\epsilon(k)$ and $W^u_\epsilon(k)$ are included in the interior of $Y_{k}$ for every $k\in K$.

Let $W^s_{loc}(\Omega)$ and $W^u_{loc}(\Omega)$ be two local stable and unstable manifolds of $\Omega$ which are tangent at $\mu= \mu_0$ and unfold non-degenerately.  Then for every periodic point $\Omega'$ close to $\Omega$, there are local stable and unstable manifolds $W^s_{loc}(\Omega')$ and $W^u_{loc}(\Omega')$ close to $W^s_{loc}(\Omega)$ and $W^u_{loc}(\Omega)$. Hence $W^s_{loc}(\Omega')$ and $W^u_{loc}(\Omega')$ are tangent at $\mu= \mu_0'$ close to $\mu_0$  and unfold non-degenerately.

As in section \ref{PTrev}, by non-degenerate homoclinic unfolding between $W^s_{loc}(\Omega')$ and $W^u_{loc}(\Omega')$ there exist:
\begin{itemize}
\item $N\ge 1$ which does not depend on $\Omega'$ close to $\Omega$,
\item a point $P$ in $W^u_{loc}(\Omega')$ and a point $Q\in W^s_{loc}(\Omega')$, with coordinates respectively $(p,0)$ and $(0,q)$ in the chart $y_{\Omega'}$,
\item  a neighborhood $D_P$ of $P$ in $Y_{\Omega'}$ sent by $f^N$ into $Y_{\Omega'}$,
\end{itemize}
 so that in the chart given by $y_{\Omega'}$ $f_{\mu}^N|D_P$ has the following form:
\[ (p+x, y)\in D_P\mapsto  (\xi x^2+\alpha(\mu-\mu'_0) +\gamma \cdot y, q+\zeta \cdot x)+ E_\mu(p+x, y)\in D\]       
where  $\alpha$, $\zeta$, $\xi$ and $\gamma$ are non-zero constants (independent of $\mu$), and $E_\mu=(E_\mu^1,E_\mu^2)\in C^{r}(\R\times \R^{2}, \R^2)$ satisfies $(\mathcal E_1-\mathcal E_2-\mathcal E_3)$. 

Hence we can apply Theorem \ref{Misurenorm}, if the period $p$ is sufficiently large (greater than a certain $M$). We obtain the conclusion of Lemma \ref{prenh} as we got remark \ref{preNH}.  
\end{proof}

\subsubsection{Proof of Theorem \ref{partialans} on the parameter domain size of Hénon renormalization}\label{app2}
Let $P_a(x)=x^2+a$. For $C>0$ and $\Lambda>1$, let $E_{C\lambda}$ be the subset of renormalization parameter intervals $\mathcal I$ defined for Theorem \ref{partialans}. 

We want to show that there exists  $b_0>0$ such that for every $\mathcal I\in E_{C\lambda}$, the domain $\hat {\mathcal I}_{b_0}$ is a closed strip which stretches across $\R\times [-b_0,b_0]$. 

From \cite{Ha11}, for every $N>0$, there exists $b_N>0$ such that for every $\mathcal I\in E_{C\lambda}$ of period less than $N$, the domain $\hat {\mathcal I}_{b_0}$ is a closed strip which stretches across $\R\times [-b_N,b_N]$. 

Hence it remains to study the case of intervals $\mathcal I$ of large periods $p$. For this end we use the system of charts defined in \textsection \ref{PSHL} for $r\ge 4$. We know that there exists $\epsilon=\epsilon(C,\Lambda)>0$ depending only on $C,\lambda$, such that with $\M$ the $\epsilon$-neighborhood of $\mathcal I\times \{0\}$, the system of charts $(\sY,\sT,\mathcal Y(f_{ab}))_{ab\in \M}$ is linear for the Hénon map $f_{ab}(x,y)=(x^2+a+y,-bx)$ and depends $C^r$-smoothly on $ab\in \M$.  

Let $a_\star\in \mathcal I$ be such that $0$ is super attracting of period $p$.  Then we have the form:
\[ y_{\sa_1{a_\star 0}}\circ f_{a_\star 0}\circ y_{\sa_0{a_\star 0}}^{-1}\colon (x,y)\mapsto (\xi x^2+\gamma y, 0)+(A(x,y),0),\]
 with $A$ satisfying $A(0)=\partial_x A(0)=\partial_x^2 A(0)=0$,  of class $C^{r-2}$-bounded and norm bounded in function of $C,\lambda$. 

Hence for $a\in \mathcal I$, we have the form:
\[ y_{\sa_1{a0}}\circ f_{a0}\circ y_{\sa_0{a0}}^{-1}\colon (x,y)\mapsto (\xi x^2+\gamma y+\alpha (a-a_\star), 0)+E_{a,0}(x,y),\]
with $E$ of class $C^{r-2}$, of bounded norm (depending only on $C,\Lambda$) and satisfying $(\mathcal E_1-\mathcal E_2-\mathcal E_3)$ at $(a,b)=(a_\star, 0)$. Also $(\mathcal E_1)$ holds for $\mu= a-a_\star$.

The charts $y_{\sa_1}$ and $y_{\sa_0}$ depends $C^{r-2}$-smoothly on $a,b$, with uniform bounds in function of $C,\Lambda$. Moreover their image is uniformly high in function of $C,\Lambda$.  

Consequently, for $\epsilon$ small only in function of $C,\Lambda$ for every $(a,b)\in \M$,  we have the form:
\[F_{ab} := y_{\sa_1{ab}}\circ f_{ab}\circ y_{\sa_0{ab}}^{-1}\colon (x,y)\mapsto (\xi x^2+\gamma y+\alpha (a-a_\star), 0)+E_{a,b}(x,y),\]
where $a\mapsto E_{ab}$ is $C^{r-2}$-close to $a\mapsto E_{a0}$, for every $\epsilon$ small only in function of $C,\Lambda$.

For every $|b|<\epsilon$, let $p_{b}\approx 0$ be such that the 1st coordinate of $\partial_x F_{a_\star b}$ vanishes at $(p_{a_\star b}, 0)$. 
By a Taylor development, there exists $q_{b}\approx 0$, $\xi_{b}\approx \xi$, $\gamma_{b}\approx \gamma$, $\alpha_{b}\approx \alpha$,  $\zeta_{b}\approx 0$, $a_b\approx a_\star$, and $(E'_{ab})_a$ $C^{r-2}$-close to $(E_{a,0})_a$ and satisfying $(\mathcal E_1-\mathcal E_2-\mathcal E_3)$ such that for $(a,b)\in \M$:
\[F_{ab} := y_{\sa_1{ab}}\circ f_{ab}\circ y_{\sa_0{ab}}^{-1}\colon (p_{b} +x,y)\mapsto (\xi_{b} x^2+\gamma_{b} y+\alpha_{b} (a-a_{b}), q_{b} +\zeta_{b} x)+E'_{ab}(x,y),\]

We recall that $Y_{\sa_0}\cap \R\times \{0\}$ 
contains the $C$-neighborhood of $0$. Also the height of  $Y_{\sa_1}(ab)$ is uniformly long. As the the norm of $\mathcal Y$ is bounded (in function of $C,\Lambda$), it comes that $p_b$ belongs to $I^u_{\sa_0}(ab)$,  $q_b$ belongs to $I^s_{\sa_1}(ab)$, if $|b|$ is smaller than a function $\epsilon$ of $C$ and $\Lambda$.

Therefore we can apply Theorem \ref{Misurenorm} to get a prerenormalization of $(f_{a,b})_a$ for every $|b|<\epsilon $.
From the form of the parameter dependence of the renormalization we have then for every $b\in(-\epsilon, \epsilon)$ the existence an interval $\mathcal I_b$ depending continuously of $b$ for which the Hénon map is renormalizable with the same period.  
We put $\hat {\mathcal I}:= \cup_{b\in (-\epsilon,\epsilon)}\mathcal I_b\times \{b\}$.

\subsubsection{Sketch of an application to a Conjecture of Hénon}\label{app3}

As mentioned in the introduction, a famous conjecture of Hénon \cite{He} states that the map : $(x,y)\mapsto (1-1.4 x^2+y,0.3 x)$  has a non-uniformly hyperbolic attractor. In dimension 1, the renormalization techniques are very useful to decide if a unimodal map is non-uniformly hyperbolic (\cite{AvilaGugu03}). It is often used along a hyperbolic chain. 

We notice that the determinant has absolute value equal to $0.3$, whereas a numerical experimentation gives that the Lyapunov exponent should be greater than $\Lambda= \exp(0.419)$. For special chart (like Pesin ones), we should be able to find a hyperbolic system of charts 
which is $\Lambda$-hyperbolic, and such that $y_{\sb\mu}^{-1}\circ f_\mu\circ y_\sa$ has norm approximately $\Lambda$ and determinant approximately $0.3$. Hence the inequality $\|Df\|^{r-3}\cdot |det\, Df|<1$ is approximately equivalent to $0.3 \cdot  \Lambda ^{r-3} <1$ and so to:
\[ r< 3-\frac{\ln\, 0.3}{0.419}\approx 5.873443446 \]
This gives some room to apply the linearization (we need only $r=4$ to get $C^2$-conjugacy) and the pre-renormalization techniques of Theorem \ref{Misurenorm} to this context.

\section{Proof of Renormalizations (Theorems \ref{PT2} and  \ref{Misurenorm})}\label{Proof of Palis-Takens }
As both proofs are the same, we will do only the one of Theorem \ref{PT2}.

We can assume that $N=1$ by considering the iterate $f^N_\mu$ instead of $f_\mu$. 

Let us simplify the expression of $f_\mu|D_P$ by doing a coordinates change with the linear map:
\[ (x,y)\mapsto (\xi\cdot x, {\xi\gamma}\cdot  y).\]
In these new coordinates $P=(\xi p,0)=:(p',0)\quad\text{and}\quad Q=(0,\xi \gamma q)=(0,q')$.
Also the expression of $f_\mu|D_P$ has the form:
\[f_\mu|D_P\colon P+ (x, y)\mapsto  Q+(x^2+ \nu+y , -b'x)+ E'_\mu (P+(x, y))\in D_Q\]
with $\nu= \xi\cdot \mu$, $-b':= \zeta\gamma$, and $E'_\mu=(A'_\mu,b'\cdot B'_\mu)$ a family of maps which satisfies $(\mathcal E_1-\mathcal E_2-\mathcal E_3)$.

For every $n\ge 0$, let $D_P'$ be the domain of definition of 
\[\phi_{\mu,n}:=(f_\mu|D)^n\circ (f_\mu|D_P).\]

We recall that on its definition domain $(f_\mu|D)^n$ is equal to:
\[L^n_\mu (x,y)=(\sigma^n_\mu x,\lambda^n_\mu y),\]
and so:
\[\phi_{\mu,n}(p'+x, y)=\Big(\sigma^n_\mu x^2+\sigma^n_\mu \nu + \sigma^n_\mu y, \lambda^n_\mu q' - b' \lambda^n_\mu x\Big)+ L^n_\mu\circ E'_\mu(p'+x, y).\]

Using the change of coordinates:
\[\left\{\begin{array}{ccc}
C_\mu\colon (x,y)&\mapsto& \big(\sigma_\mu^n ( x-p'), \sigma_\mu^{2n}  (y-r)\big)\\
C_\mu^{-1}\colon (x,y)&\mapsto& \big( \sigma_\mu^{-n} x+p', \sigma_\mu^{-2n}  y+r\big)\end{array}\right\} \quad \text{with }r:= \lambda_\mu^n q'.\]
It holds for $(x,y)\in  C_\mu(D_P')$: 
\begin{multline*}\phi_{\mu,n}\circ C_\mu^{-1}  (x,y)= \Big(\sigma_\mu^{-n} x^2+\sigma_\mu^n \nu + \sigma^{-n}_\mu y +\sigma_\mu^n r, \lambda_\mu^nq' -b' \lambda_\mu^n\sigma_\mu^{-n} x\Big)+L_\mu^n\circ E'_\mu\circ  C_\mu^{-1}(x,y)\end{multline*}
and so :
\begin{multline*}C_\mu\circ \phi_{\mu,n}\circ C_\mu^{-1}  (x,y)= \Big(x^2+\sigma_\mu^{2n} \nu +  y +\sigma_\mu^{2n} r-\sigma_\mu^np', \lambda_\mu^n\sigma_\mu^{2n}q'-\sigma_\mu^{2n}r-b' \lambda_\mu^n\sigma_\mu^{n} x\Big)\\
+ DC_\mu\circ  L_\mu^n\circ E'_\mu\circ  C_\mu^{-1}(x,y)\end{multline*}

We reach the form:
\[\mathcal R f_a:= (x,y)\mapsto \Big(x^2+a +  y , -b x\Big) + (A_a,b\cdot B_a)(x,y)\]
with  $b= b' \lambda^n_0\sigma_0^{n}$, $a= \sigma_\mu^{2n} \nu+\sigma_\mu^{2n}\lambda_\mu^n q'-\sigma_\mu^np'$, and:
\[A_a(x,y) =\sigma_\mu^{2n} A_\mu'\circ C_\mu^{-1}(x,y),\quad B_a(x,y)=\frac{\lambda_\mu^n\cdot \sigma_\mu^{2n}}{\lambda_0^n\cdot \sigma_0^{n}}  B_\mu'\circ C_\mu^{-1}(x,y)+(1-\frac {b'}b \lambda_\mu^n\cdot \sigma_\mu^{n})x \]

The family $\mathcal R f_\mu$ is the \emph{renormalization} of $f_\mu$. 

\begin{prop}\label{PTestimate} The family $(\mathcal R f_a)_a$ is Hénon $C^r$-like.
Moreover, for every $\delta>0$, if $n$ is large, this family is Hénon $C^r$-$\delta$-like. 
\end{prop}
The proof of this Proposition is purely analytical and so it is postponed to section \ref{proofofpropPTestimate}.

\section{Proof of  normal forms of system of charts}
\subsection{Pseudo invariant stable foliation (Proof of Proposition \ref{foliation})}\label{Pseudo invariant stable foliation}
Let $f\in \M$ be satisfying the hypotheses of Proposition \ref{foliation}.

 The field $e(f)$ is found as an invariant, $(r-1)$-normally expanded graph of a function 
in $C^r(W\times \M, \P\R^1)$,  where $\P\R^1$ is the real one-dimensional projective space. Let us first construct the dynamics 
 which will normally expand this graph. First let us notice that $(\mathcal D_r)$ remains valid on a small neighborhood $V_1$ of $cl(V)$.

Let $u$ be a smooth, unit vector field in $\chi$ over $W$. We observe that for every $z\in V$, the inner product $D_zf(u(z))\cdot u(f(z))$ is not zero. Let $sgn(z)\in \{-1,1\}$ be its sign. 
Let $u_\bot$ be a smooth, unit vector field orthogonal to $u$ over $W$. 

We denote by $u^* $ the one form equal to $1$ at $u$ and zero on its orthogonal complement. Similarly we define the form $u_\bot^{*}$.
Let $V_2$ be a small neighborhood of $V_1$ in $W$ and $\rho_1\in C^\infty(W, [0,1])$ a bump function equal 
to $1$ at $V$ and $0$ on $V_1^c$, and  $\rho_2\in C^\infty(W, [0,1])$ a bump function equal 
to $1$ at $V_1$ and $0$ on $V_2^c$. We extend continuously $sgn$ to $V_2$ with value in $\{-1,1\}$. 
We notice that $D_zf:=  u\cdot u^* \circ D_zf +  u_\bot\cdot u_\bot^* \circ D_zf$. Put:
\[\mu(z):= \sup_{\|v\|=1}\frac{|u_\bot^* \circ D_zf(v)|^2}{|u^* \circ D_zf(v)|^2} \quad \text{and}\quad \tau(z):= \sqrt{(1-\rho_1^2(z) )\mu(z)+1}.\]  
 \[ \phi_f(z):=   sgn(z)\cdot\tau(z) \cdot \rho_2(z)\cdot  u\cdot u^* \circ D_zf +  (1-\rho_2(z))\cdot   u\cdot u^*(z)+sgn(z)\cdot \rho_1(z)\cdot  u_\bot\cdot u_\bot^* \circ D_zf.\]
 
We notice that on $V_1$, it holds $ \phi_f(z)=  sgn(z)\cdot \tau(z)\cdot  u\cdot u^* \circ D_zf +sgn(z)\cdot \rho_1(z) u_\bot\cdot u_\bot^* \circ D_zf$. Hence it holds on $V_1$:
\[|det(\phi_f(z))|= \rho_1(z)\tau(z) |det( D_zf)|\le |det( D_zf)|
\quad \text{and}\quad 
\|\phi_f(z)\|\ge  \|Df(z)\|\]
Indeed a computation gives that for every $z$, $\rho_1(z)\tau(z) \le 1$. Moreover, since the vectors in $\chi$ are the most expanded (see remark \ref{rema2.9}), there exists a constant $C'>0$ such that for every $k$, for every $z\in \bigcap_{i=0}^{k-1} f^{-i}(V_1)$, it holds that:
\[\|\phi_f^k (z,u)\|\ge C'\|D_zf^k\|.\]

On the complement of $V_1$, we have $\phi_f(z):=  sgn(z)\cdot \tau(z)\cdot \rho_2(z) \cdot u\cdot u^* \circ D_zf +  (1-\rho_2(z))  \cdot u\cdot u^*(z)$. Then it holds: 
\[|det(\phi_f(z))|=0\;.\]
Consequently, it holds for every $-1\le i \le r-3$, for every $k\ge 0$, every $z\in W$:
\[ |det(\phi^k_f(z))|\cdot \|\phi_f^k (u(z))\|^{-2}\cdot \|D_z f^k\|^{i+2}\le C'^{-2} |det(\phi^k_f(z))|\cdot \|D_z f^k\|^{i}\]
By  $(\mathcal D_r)$ of remark \ref{rema2.9}, it comes for every $1\le i\le r-1$: 
 \begin{equation}\label{Drpri} |det(\phi^k_f(z))|\cdot \|\phi_f^k (u(z))\|^{-2}\cdot \|D_z f^k\|^{i}\le C'^{-2}(C\Lambda^k)^{-1}\end{equation}

The action $\tilde \phi_f$ of the map $\phi_f$ on $\cup_{z\in W} \P\R^1\cap \chi(z)$ is  well defined and of class $C^r$. Moreover it sends the closure $\chi(z)$ into  $\chi(f(z))$, and so $\tilde \phi_f|\chi$ is contracting for the Hilbert metric associated to $\chi$. However, we cannot conclude to a fixed section of $\chi$ since $f$ is possibly not bijective on the base $W$. 

Moreover $\tilde \phi_f$ is not well defined on all $\P\R^1$ due to the singularities. 

Nevertheless the action $\tilde \phi_f(z)^{-1}$ of the map $\phi_f(z)^{-1}$ is well defined on $\P\R^1\setminus \chi({f(z)})$, and sends it into $\chi^c(z):=\P\R^1\setminus \chi({z})$, for every $z$. Moreover the following map is of class $C^{r-1}$.
 \[W\times \M\times \P\R^1\supset \bigsqcup_{(z,f)\in W\times \M}\chi^c(z)\ni (z,f,v) \mapsto \tilde \phi_f^{-1}(z)(v)\in \bigsqcup_{(z,f)\in W\times \M}\chi^c(z)\subset W\times \M\times \P\R^1\]

Let $\Gamma$ be the space of $C^{0}$-maps $W\times \M\ni (z,f)\mapsto v(z,f)\in \R\P^1$ such that $v(z,f)$ is in the complement of $\chi(z)$ for every $z\in W$.
Let us consider the graph transform:
\[\Phi\colon \Gamma \ni v \mapsto [(z,f)\mapsto \tilde \phi_f^{-1}(z)\circ v(f(z),f)]\in \Gamma.\]
By the mapping cone property of $\chi$,  the map $\Phi$ has a contracting iterate for the $C^0$-distance.  Let us denote its fixed point by:
\[W\times \M\ni (z,f)\mapsto e(f)(z)\in \P\R^1\]
The following Claim  implies immediately Proposition \ref{foliation}.
\begin{Claim}\label{regularitedee} The map $(z,f)\mapsto e(f)(z)$ is of class $C^{r-1}$, with norm bounded in function $C,\Lambda$, and the $C^r$-norms of $f\in \M$. 
\end{Claim}

Claim \ref{regularitedee} is a consequence of the following contraction Lemma on the space \[\Sigma_{r'}:=\{v\in C^{r'}(W,\R\P^1):\; v(z)\notin \chi(z)\; \forall z\in W\},\quad r'<r.\]

\begin{lemm}\label{contractionpartielle}
For all $f\in  \mathbb M$, the following operator is of class $C^\infty$:
\[\Phi_f\colon \Sigma_{r-1} \ni v \mapsto [z\mapsto \tilde \phi_f^{-1}(z)\circ v(f(z))]\in \Sigma_{r-1}.\]
Moreover  it has an iterate which is contracting on the intersection of the $C^0$-neighborhood of $e_f$ with $\Sigma_{r-1}$.
\end{lemm}

This lemma is shown below and will be generalized to the following contracting lemma on the space
\[\Sigma:= \{v\in C^{r-1}(W\times \mathbb M,\mathbb R\mathbb P^1):\; v(z,f)\notin \chi(z)\forall z\in W\times \mathbb M\}\]
\begin{lemm}\label{Lemma 5.3}
The following operator has a contracting iterate on the intersection of the $C^0$-neighborhood of $(z,f)\mapsto e_f(z)$ :
\[\Psi\colon \Sigma\ni v\mapsto [(z,f)\mapsto \tilde \phi^{-1}_f (z)\circ v(f(z))]\in\Sigma\]
\end{lemm}
%
%
%
In order to prove Lemma \ref{contractionpartielle}, we shall simplify the implements. First let us notice that $e_f$ is invariant by $\phi_f$ and remains away of $cl(\chi)$, which is sent into itself  by $\phi_f$. Let $\tilde e_f$ be a unit vector field in $e_f$. Hence, there exists $C''>0$ such that for every $k\ge 0$:
\[ \|\phi_f^k(z)(e_f(z))\|\cdot \|\phi_f^k(z)(u(z))\|\le C'' |det(\phi_f^k(z))|\]
By (\ref{Drpri}), it holds for every $k\ge 0$, for every $i\le r-1$, for every $z\in W$: 
 \begin{equation}\label{Drpri2}  \|\phi_f^k(z)(\tilde e_f(z))\|\cdot \|\phi_f^k (u(z))\|^{-1}\cdot \|D_z f^k\|^{i}\le C''(C')^{-2}(C\Lambda^k)^{-1}\end{equation}
For an equivalent norm $\|\cdot \|_z$ of $\mathbb P\R^1$, we can assume that $u(z)$ and $\tilde e_f(z)$ are orthogonal and unitary. This induces a Riemannian metric on $\mathbb P\R^1$, and so an equivalent metric on $\mathbb P\R^1\times W$. 

To make the notations less cumber-storming, we shall work with an iterate of $\phi_f$ and $f$, such that we can assume $C=C'=C''=1$, such that $\Lambda' \in (1,\Lambda)$:
\begin{equation}\tag{$\mathcal D_r'$}
\frac{\|\phi_f(\tilde e_f(z))\|}{\|\phi_f (u(z))\|}\cdot \|D_z f\|^{i}\le \Lambda'^{-1},\quad \forall 1\le i\le r-1,\quad \forall z\in W
\end{equation}

A consequence of ($\mathcal D_r'$) is the following:
\begin{lemm}\label{corocontract}
\[\|\partial_v (\tilde \phi_f^{-1}(z))( e_f)\| \cdot \|D_z f\|^i\le \Lambda'^{-1},\quad \forall 1\le i\le r-1\]
\end{lemm}
\begin{proof} If $Df|e_f(z)= 0$, then the derivative  $\partial_v (\tilde \phi_f^{-1}(z))(e_f)$ equals zero. Otherwise, since $e_f(z)$ and $u(z)$ are orthogonal, $\|\partial_v (\tilde \phi^{-1}_f(z))(e_f)\|$ is equal to 
\[\lim_{\theta\to 0} \left\|\frac{\phi_f (\tilde e_f+\theta u)}{\theta\|\phi_f (\tilde e_f+\theta u)\|}-\frac{\tilde e_f}\theta\right\|^{-1}= \frac{\|\phi_f (\tilde e_f)\|}{\|\phi_f (u)\|}.\]
We conclude by using $(\mathcal D_r')$.
\end{proof}

\begin{proof}[Proof of Lemma \ref{contractionpartielle}]
By induction on $i\le r'$, we are going to show the contraction of $\Phi_f$ for an equivalent norm to $\max_{k\le i}  \|\partial_z^k v\|$ for $i\le r$, on the intersection of $\Sigma_{r'}$ with a $C^0$-neighborhood of $e_f$. Let $v\in \Sigma_{r'}$ be $C^0$-close to $e_f$.

By the  Leibniz formula, we have $\partial_z^i(\tilde \phi_f^{-1}(z)\circ v\circ f(z))$ equals to
\[(\partial_z^{i}(\tilde \phi_f^{-1}(z)))\circ v\circ f(z)
+ \sum_{0\le k<i} \binom ik \partial_z^{k}\partial_v (\tilde \phi_f^{-1}(z))\circ \partial_z^{i-k} (v\circ f(z)).\]

By changing the metric of  $W$, we can suppose that $\partial_z^{i}(\tilde \phi_f^{-1}(z))$ is $C^0$-small for all $z,f$. As the map $\tilde \phi_f^{-1}(z)$ is a Moebius transform of $\P\R^1$   for every $f$ and  $z$, it comes that $\partial_z^{i}(\tilde \phi_f^{-1}(z))$ is $C^1$-small and so $v\mapsto \partial_z^{i}(\tilde \phi_f^{-1}(z))\circ v\circ f$ contracts the $C^{0}$-norm.

By the  Faa-di-Bruno formula, we have:
\[\partial_z^{i} (v\circ f(z)) = \sum_{l=1}^i 
(D_z^l v)\circ f(z) \circ 
Q_{li} ( D_z f(z), \dots ,D_z^{i-l+1} f(z)),\]
with $Q_{li} \in \Z[X_1,\dots ,X_{i-l+1}]$, $Q_{ii}(X_1)=X_1^k$.  

The same ``metric-change'' argument as for the $C^0$-contraction shows the contraction of an iterate of:
\[D^l_z v\mapsto \binom ik \partial_z^k \partial_v (\tilde \phi^{-1}_f(z))D_z^lv\circ f\circ Q_{l(i-k)}( D_z f(z), \dots ,D_z^{i-k-l+1} f(z)),\quad \forall i-k> l\ge 1\]

Hence, by induction and  by changing the norm $N_{i}$ to an equivalent norm of the form $\max_{k\le i}  c_{k}\|\partial_z^k v\|_{C^0}$, with adequate coefficients $c_{k}>0$, it remains only to prove the contraction of 
\[\partial_v (\tilde \phi_f^{-1}(z))\circ (D_z^{i}v)(D_zf)^i,\]
which is indeed $\Lambda'^{-1}$-contracting by Lemma \ref{corocontract}.
\end{proof}
\begin{proof}[Proof of Lemma \ref{Lemma 5.3}]
The proof is done in the same way as for Lemma \ref{contractionpartielle}. We look at the derivative of $\partial^i_z\partial_f^j(\tilde \phi_f^{-1}(z)\circ v\circ f(z))$, for which the formula is similar to $\partial_z^{i+j}(\tilde \phi_f^{-1}(z)\circ v\circ f(z))$. Indeed, $z$ and $f$ play a similar role, in particular, we can change the metric on $\M$ in order to prove that the contraction of the above derivatives is equivalent to the contraction of 
\[v\mapsto (\partial_v (\tilde \phi_f^{-1}(z))\circ (\partial^i_z \partial ^j_f v)(D_z f)^i)\]
for an equivalent norm of the form $\max_{i'+j'\le i+j} c_{i'j'} \|\partial_z^{i'} \partial_ f^{j'} v\|_{C^0}$. 
\end{proof}

\subsection{Linearization of unstable directions (Proof of Proposition \ref{semilinearization})}\label{a2B}

The proof is rather classical (see \cite{Guguyoyo10} Appendix A, or \cite{Su87}), if no parameter dependence. Hence we redo the proof here.  

For $\st = (\sa,\sb)\in \sT$, put $f_{\st\mu}:= F_{\st\mu}|I^u_{a\mu }\times \{0\}$. 

First we fix $\mu$ and we want to linearize the functions $(f_{\st\mu})_\st$ which are of class $C^{r}$.
 
Let $\Gamma_\mu$ be the space $\{(\phi_{\sa})_{\sa\in \sY} \in \prod_{\sa\in \sY} C^{r}(I^u_{\sa\mu}, \R):\; \phi_{\sa}(0)=0\; \text{and}\; D\phi_{\sa}(0)=1\}$. 
It is a convex subset of the Banach vector space $\prod_{\sa\in \sY} C^{r}(I^u_{\sa\mu}, \R)$ equipped with the norm:
\[\|(\phi_{\sa})_{\sa}\|_{C^{r}}= \max_{0\le s\le r,\sa\in \sY} \|D^s \phi_{\sa}\|.\]  

We want show that the following map has a fixed point  in $\Gamma$:
\[\Phi_\mu\colon \Gamma_\mu\ni 
\phi=(\phi_{\sa})_{\sa}
\to (\Phi_{\mu}(\phi)_\sb)_{\sb}\in\Gamma_\mu,\quad \text{with  } \left\{\begin{array}{cl}
\Phi_{\mu}(\phi)_\sb = D_0f_{\st \mu}\cdot  \phi_{\sa}\circ f_{\st \mu}^{-1}& \text{if }\exists \st= (\sa,\sb)\in \sT ,\\
\Phi_{\mu}(\phi)_\sb=id &\text{otherwise.}\end{array}\right.\]
A  fixed point $(\phi_{\sa\mu})_\sa$ of $\Phi_\mu$ will satisfy for every $\st=(\sa,\sb)\in \sT$ that 
\[D_0f_{\st \mu}\cdot  \phi_{\sa\mu}\circ f_{\st \mu}^{-1}=\phi_{\sb\mu}
\Leftrightarrow 
D_0f_{\st \mu}\cdot id=\phi_{\sb\mu} \circ f_{\st \mu}\circ \phi_{\sa\mu}^{-1}\]
And so the charts $y_{\sa\mu}':= (x,y)\mapsto y_{\sa\mu}(\phi_{\sa}^{-1}(x),y)$ are of type $A$ for $f_\mu$.

Hence to prove Proposition \ref{semilinearization}, it is sufficient to prove the existence of such a $C^{r}$-fixed point $\phi_{\mu}=(\phi_{\mu\sa})_\sa$ with bounded norm and to show that the parameter dependence is sufficiently smooth so that the following is bounded
\[\sup_{\sa} \left(\|(x,\mu)\mapsto \phi^{-1}_{\mu\sa}(x)\|_{C^{1}}, \max_{x,\mu, (i,0,k)\in \Delta^r_A}\|\partial_x\partial_\mu \phi_{\mu\sa}(x)\|\right)\]
  In order to do so, first we show that $\Phi_\mu$ has a contracting iterate.

\paragraph{Proof that $\Phi_\mu$ has a contracting iterate for $r=2$.}
For any $\phi$ and $\psi$ in $\Gamma_\mu$, the functions $\Phi_{\mu}(\phi)$ and $\Phi_{\mu}(\psi)$ have coordinatewise the same value and derivative at $0$. Thus the $C^2$-contraction of an iterate of $\Phi_{\mu}$ is equivalent to the uniform contraction of the second derivative by an iterate of $\Phi_{\mu}$.
Moreover, since $(f_{\sa\mu})_{\sa}$ is uniformly expanding and since the interval $(I^u_{\sa\mu})_\sa$ are uniformly bounded, 
to get the contraction of an iterate of $\Phi_\mu$ it is sufficient to prove the existence of $\delta>0$ such that $\Phi_{\mu}$ contracts the following semi-norm on $\Gamma_\mu$:
\[N((\phi_\sa)_\sa)=\sum_{\sa\in \sY,\; x\in I_{\sa\mu}^u\cap [-\delta, \delta] } \|D^2_x \phi_\sa\|.\]
Indeed, the sets $I_{\sa\mu}^u\setminus [-\delta, \delta]$ escape in finitely many iterations under the action of $(f_{\st\mu})_\mu$, by condition $(\mathcal M)$ and the boundedness of $(|I^u_{\sa\mu}|)_{\sa\in \sY}$.

  We compute: 
\[D^2 (\Phi_{\mu}(\phi)_\sb)(x)=  
D_0f_{\st \mu}\cdot D^2\phi_{\sa}\circ f_{\st \mu}^{-1}(x)\cdot  (Df_{\st\mu}^{-1}(x))^2
+D_0f_{\st \mu}\cdot D\phi_{\sa}\circ f_{\st \mu}^{-1}(x)\cdot D^2f_{\st\mu}^{-1}(x).\]

For $\delta$ small, for $x\in f_{\st\mu}^{-}(I_{\sb\mu}^u\cap [-\delta, \delta])$ the derivative $Df_{\st\mu}(x)$ is close to $Df_{\st \mu}(0)$, and this uniformly in $\sa\in \sY$ since $\mathcal Y$ is of bounded norm. Also, by rescaling the interval $(I^u_{\sa\mu})_{\sa\in \sY}$ by  factors (as we do to construct adapted metric), we can assume that for every $\st=(\sa,\sb)\in \sT$, for every $x\in I_{\sb\mu}^u\cap [-\delta, \delta]$,  we have $|(Df_{\st\mu}^{-1}(x))^2\cdot Df_{\st \mu}(0)|<\Lambda^{-1/2}$. Such factors are bounded in function of $C$, $\Lambda$ and the $C^2$-norms of $f_\mu$ and $\mathcal Y$. 

 Hence given $\phi,\psi\in \Gamma_\mu$, we have for all $\sa\in \sY$, $x\in I_{\sa\mu}^u\cap [-\delta, \delta]$:
\[|D^2(\Phi_{\mu}(\phi)_\sb-\Phi_{\mu}(\psi)_\sb)(x)|\le \Lambda^{-1/2}\cdot   
N((\phi_\sa)_\sa-(\psi_\sa)_\sa)
+|D_0f_{\st \mu}| \cdot |D^2f_{\st\mu}^{-1}(x)|\cdot 
|(D\phi_{\sa}-D\psi_{\sa})\circ f_{\st \mu}^{-1}(x)|.\]
On the other hand, for $x\in I_{\sa\mu}^u\cap [-\delta, \delta]$:
\[\|(D\phi_{\sa}-D\psi_{\sa})(x)\|= \left|\int_{0}^x 
(D^2\phi_{\sa}-D^2\psi_{\sa})(t)dt\right|\le \delta N((\phi_\sa)_\sa-(\psi_\sa)_\sa)\]
Consequently:
\[N(\Phi_{\mu}(\phi)_\sb-\Phi_{\mu}(\psi)_\sb)\le (\Lambda^{-1/2}+\delta\sup_{x,\st\in \sT} |D_0f_{\st \mu}| \cdot |D^2f_{\st\mu}^{-1}(x)|)\cdot   
N((\phi_\sa)_\sa-(\psi_\sa)_\sa).\]
Therefore, for $\delta$ small enough, $\Phi_{\mu}$ is contracting for $N$. Hence it has a fixed point whose norm is bounded depending on the $C^{2}$-norm of $f_\mu$, $\mathcal Y$ and $C,\Lambda$.

\paragraph{Proof that $\Phi_\mu$ has a  contracting iterate for $r>2$.}
By induction, an iterate of $\Phi_\mu$ contracts the $r-1$-first derivatives.
 Let us show that it contracts also the $r^{th}$-derivative.
 The Faa-di-Bruno formula is:
\[D^r (\phi_{\sa}\circ f_{\st \mu}^{-1})= \sum_{i=1}^r D^i \phi_{\sa} \circ f_{\st \mu}^{-1} Q_i (Df_{\st\mu}^{-1},\dots ,D^{r-i}f_{\st\mu}^{-1}),\]
with $Q_i \in \Z[u_1,\dots ,u_{r-i}]$, $Q_r(u_1)=u_1^r$. 

By changing the norm $\|\phi_{\sa}\|_{C^r}$ to an equivalent norm of the form 
\[\max_{i\le r} c_i \|D^i \phi_{\sa}(x)\|\]
 (with suitable $c_i>0$), to show that $\Phi_\mu$ has a contracting iterate, 
it is sufficient to remark as above that:  
\[ \sup_{I^u_{\sa\mu}\cap [-\delta,\delta]} |D f_{\st\mu}(0) \circ D^r \phi_{\sa} \circ f_{\st\mu}^{-1} \cdot Q_r (D^lf_{\st\mu}^{-1})_l(x)|
=\sup_{I^u_{\sa\mu}\cap [-\delta,\delta]} |D f_{\st\mu}(0) \circ D^r \phi_{\sa} \circ f_{\st\mu}^{-1} \cdot (Df_{\st\mu}^{-1})^r (x)|\]\[
\le \sup_{I^u_{\sa\mu}\cap [-\delta,\delta]} \Lambda^{-r+1} |D^r \phi_{\sa}(x)|\]

\paragraph{Parameter dependence of $(\rho_{i\mu})_\mu$}

For every $\sa\in \sY$, let $I_\sa:= \cup_{\mu\in \M} I_{\sa\mu}\times  \{\mu\}$ and let $B_r$ be the space of family of maps $(\phi_\sa)_{\sa\in \sY}$ with continuous derivative following $\partial^p_\mu \partial^q_x$ among  $p+q\le r$ with $p\notin \{r-1,r\}$. This is a Banach space endowed with the norm:
\[\|(\phi_\sa)_{\sa\in \sY} \|= \max_{p+q\le r, \, p\le r-2, \sa\in \sY} \|\partial^p_\mu \partial^q_x \phi_i \|.\]  

Let $\Gamma$ be the convex subset of $B_r$ formed by the maps $(\phi_i)_i\in B_r$ such that $(\phi_i(\cdot, \mu))_i$ is in $\Gamma_\mu$ for every $\mu\in \M$. 
 
We notice that the following map is well defined:
\[\Phi\colon \Gamma\ni (\phi_{i})_{i=1}^N\longmapsto \big[(x,\mu)\mapsto (\Phi_{\mu}(\phi(\cdot,\mu )_{i})_{i=1}^N\big]\in\Gamma\]
We want to prove the existence of a fixed point of $\Phi$ with a bounded norm. This implies that $(\mathcal Y'(\mu))_\mu$ is a $C^{r-1}$-parameter dependent systems.

For this end, we follow the same proof as above, by showing that it has a contracting iterate. 
To prove the contraction of the derivatives of the form $\partial^p_\mu \partial^q_x \phi_i $, with  $p+q\le r$ and $p\notin \{r-1,r\}$, we follow the same scheme as above. We already did the case $p=0$.

For every $\mu\in \M$ and $\st=(\sa,\sb)\in \sT$, we have
$\partial_\mu\Phi(\phi)_{\sa\mu }(0)=0$ and $\partial_\mu \partial_x\Phi(\phi)_{\sa\mu }(0)=0$.  Hence it is sufficient to prove the eventual contraction on the uniform norm of the higher derivatives. We compute:
  \[\partial_\mu \Phi(\phi)_{\sb\mu}(x)=\partial_\mu D f_{\st\mu}(0)\cdot \phi_{\sa\mu }( f_{\st\mu}^{-1}(x))
+D f_{\st\mu}(0)\cdot (\partial_\mu \phi_{\sa \mu})(f_{\st\mu}^{-1}(x))
+D f_{\st\mu}(0)\cdot  (\partial_x \phi_{\sa \mu})\circ \partial_\mu f_{\st\mu}^{-1}(x)\]
For an equivalent norm on $Diff^r(\R^2)$, we can suppose that $\partial_\mu f_{\mu\sb}^{-1}(x)$ and $\partial_\mu D f_{\st\mu}(0)$ are small. 

Hence to prove that $\Phi$ contracts the derivative  $\partial_\mu \partial^q_x \phi_{\sa \mu} $ for $q\le r$, we just need to prove that the following map contracts the uniform norm of the $q^{th}$-derivative $\partial_x^q$, for $2\le q\le r$:
\[\partial_\mu \phi_{\sa\mu}\mapsto D f_{\st\mu}(0)\cdot (\partial_\mu \phi_{\sa\mu})(f_{\st\mu}^{-1}(x))\]
The proof of which is the same as above.

Similarly, for $p\le r-2$, to prove that $\Phi$ contracts the derivative  $\partial_\mu^p \partial^q_x \phi_{\sa \mu} $, with $p+q\le r$, we use the  Faa-di-Bruno formula to compute $\partial_\mu^p  \phi_{\sa \mu} $, and show that this is equivalent to prove that the following contracts the uniform norm of the $q^{th}$-derivative $\partial_x^q$, for $2\le q\le r-p$:
\[\partial^p_\mu \phi_{\sa\mu}\mapsto D f_{\st\mu}(0)\cdot (\partial^p_\mu \phi_{\sa\mu})(f_{\st\mu}^{-1}(x))\]
We remark that $[2, r-p]$ is empty for $p\in \{r-1,r\}$. This is why we lose the corresponding derivatives.
   
\subsection{Full linearization (Proof of Proposition \ref{linearization})}\label{proofof{estimaterenor}}
Let $(f_\mu)_\mu$ be a Hénon $C^{r}$-like family satisfying condition $(\mathcal D_r)$ and $(\mathcal M)$ with constants $(C,\Lambda)$ for a $C^r_B$-system of charts $(\sY,\sT,\mathcal Y(\mu))_\mu$ with bounded norm.

We recall that this implies for every $\st=(\sa,\sb)\in \sT$, that  we have the form 
\[F_{\st\mu}= y_{\sb\mu}^{-1}\circ f_{ \mu}\circ y_{\sa\mu}\colon (x,y) \mapsto (\sigma_{\st\mu}\cdot x, g_{\st\mu}(x,y)),\] 
with $(x,y,\mu)\mapsto g_{\st\mu}(x,y)$ of class $C^{r}_B$, that is the following derivatives exist continuously:
\[(x,y,z)\mapsto \partial_x^i\partial_y^j\partial_\mu^kg_{\st\mu}(x,y),\]
\[\forall (i,j,k)\in \Delta_B^r := \big\{(i,j,k);\; i+j+k\le r,\; i\le r-1,\; k\le r-2\big\}\]  

In order to linearize $g_{\st\mu}$ we show below the following.
\begin{lemm}\label{Lemmadederive}There exists an equivalent $C^r_B$-bounded system of charts $(\sY,\sT,\mathcal Y'(\mu))_\mu$ with the following form, 
for every $\st=(\sa,\sb)\in \sT$: 
\[F_{\st\mu}= y_{\sb\mu}^{-1}\circ f_{ \mu}\circ y_{\sa\mu}\colon (x,y) \mapsto (\sigma_{\st\mu}\cdot x, g_{\st\mu}(x,y)),\]
such that  there exists of $\lambda_{\st}\in \R$ satisfying $\partial_y g_{\st\mu}(x,0)=\lambda_\st$,  for every $x\in I^u_{\a\mu}$. Also
the new charts in $\mathcal Y'$ and their inverses have continuous derivatives $\partial^i_x\partial^j_y \partial_\mu^k$, for all 
$i+j+k\le r$, such that $i\le r-1$, $k\le r-2$.
\end{lemm}

We are now ready to prove Proposition \ref{linearization}. We are going to use the same linearization argument as for Proposition \ref{semilinearization} .  However we are going to linearize every $y$-coordinate along the line $x=constant$. Hence $x$ will play the same role as $\mu$.

For technical reasons, we shall chose $\eta>0$ small, so that for every $\st=(\sa,\sb)\in \sT$, with 
$\hat I ^u_{\sa \mu}$ the $\eta$-neighborhood of $I ^u_{\sa \mu}$, we can extend $F_{\st,\mu}$ to a function of the form
\[\hat F_{\st\mu}\colon \hat I ^u_{\sa \mu}\times I ^s_{\sa \mu}\in 
(x,y)\mapsto (\hat f_{\st\mu}(x), \hat g_{\st\mu}(x,y))\in \hat I ^u_{\sa \mu}\times I ^s_{\sa \mu},\]
such that
\begin{itemize}
\item  $\partial_y \hat F_{\st\mu}(x,0)=\lambda_\st$ and $\partial_y \hat F_{\st}(x,y)$  of the order of $\lambda_\st$, for all $(x,y)\in \hat I^u_{\sa \mu}\times I^s_{\sa \mu}$,
\item its derivatives $\partial^i_x\partial^j_y\partial^k_\mu$ for $i+j+k\le r$, so that $i\le r-1$ and $k\le r-2$, are continuous and bounded, this uniformly on $\sa$ and $(x,y,\mu)\in \hat I^u_{\sa \mu}\times I^s_{\sa \mu}\times \M$. 
\item $|D\hat f_{\st\mu}|\le |\sigma_{\st\mu}|$.
\end{itemize}

We notice that $\hat f_{\st\mu}|I^u_\mu= \sigma_{\st \mu} \cdot id$.

Let $B_\mu$ be the Banach space of families $(\phi_{\sa})_{\sa\in \sY}$ where $\phi_{\sa}$ is from $\hat I^u_{\sa \mu} \times \hat I^s_{\sa\mu}$ into $\R$, with continuous derivatives of the form $\partial^i_x \partial^j_y$ for every $i+j\le r$ such that $i\le r-2$, equipped with norm:
\[\|(\phi_{\sa})_{\sa}\|= \sup_{i+j\le r,\; i\le r-2,\;  \sa\in \sY} \|\partial^{i}_x \partial^j_y \phi_{\sa}\|.\]

Let $\Gamma_\mu $ be the convex subset of $B_\mu$ made by families $(\phi_{\sa})_{\sa}$ satisfying:
\[\phi_{\sa}(x, 0)=0\; \text{and}\; \partial_y \phi_{\sa}(x, 0)=1,\; \forall x\in  \hat I^u_{\sa \mu}, \;\forall \sa\in \sY.\]

The idea is to find a fixed point in $\Gamma_\mu$ of the following operator:
\[\Phi_{\mu }\colon \Gamma_\mu \ni 
\phi=(\phi_{\sa})_{\sa}\to (\Phi_{\mu x}(\phi)_\sa)_{\sa}\in\Gamma_\mu,\]
\[\text{with  }     \left\{\begin{array}{cl}
\Phi_{\mu}(\phi)_\sa = \lambda_{\st \mu}^{-1}\cdot  \phi_{\sb}(\hat f_{\st\mu}(x), \hat g_{\st\mu}(x,y))& \text{if }\exists \st= (\sa,\sb)\in \sT,
\\
\Phi_{\mu}(\phi)_\sb=id &\text{otherwise.}\end{array}\right.\]
which is on each curve $x=cst$ the same as in Proposition \ref{semilinearization} up to changing $x$ to $y$.

Also it contracts all the derivative of the form $\partial^j_y$ for $j\le r$, by the same proof as for Proposition \ref{semilinearization}. Indeed Condition $(\mathcal D_r)$ implies that $F_{\st \mu}$ contracts the $\partial_y$ direction. 

For all $\st=(\sa,\sb)\in \sT$ and $x\in I^u_{\sa\mu}$, we have
$\partial_x\Phi_\mu(\phi)_{\sa }(x,0)=0$ and $\partial_\mu \partial_y\Phi_\mu(\phi)_{\sa}(x,0)=0$.  Hence it is sufficient to prove the eventual contraction on the uniform norm of the higher derivatives. We compute:
  \[\partial_x \Phi(\phi)_{\sb\mu}(x,y)= 
  \lambda_{\st \mu}^{-1}\cdot  \partial_x \phi_{\sb}(\hat f_{\st \mu}(x), \hat g_{\st\mu}(x,y))Df_{\st \mu}(x)+  \lambda_{\st \mu}^{-1}\cdot  \partial_y \phi_{\sb}(f_{\st \mu}(x), g_{\st\mu}(x,y))\partial_x g_{\st\mu}(x,y)\]
Hence to prove that $\Phi$ contracts the derivatives  $\partial_x \partial^j_y \phi_{\sa \mu} $ for $j\le r-2$, we just need to prove that the following map contracts the uniform norm of the $j^{th}$-derivative $\partial_y^j$, for $2\le j\le r-1$:
\begin{equation}\label{map1}
\partial_x \phi_{\sb}\longmapsto[(x,y)\to\lambda_{\st \mu}^{-1}\cdot  \partial_x \phi_{\sb}(f_{\st \mu}(x), g_{\st\mu}(x,y)) Df_{\st \mu}(x)]\end{equation}
By the same proof as in Proposition \ref{semilinearization}, we can show that the map:
\begin{equation}
\partial_x  \phi_{\sb}\longmapsto[(x,y)\to\lambda_{\st \mu}^{-1}\cdot  \partial_x \phi_{\sb}(f_{\st \mu}(x), g_{\st\mu}(x,y))]\end{equation}
is $\lambda_{\st}$ contracting (for an equivalent norm). As $|Df_{\st \mu}|\le \sigma_{\st}$, by $(\mathcal D_r)$ the map  (\ref{map1}) contracts the uniform norm of the $j^{th}$-derivative $\partial_y^j$, for $2\le j\le r-1$.

Similarly, for $i\le r-2$, to prove that $\Phi$ contracts the derivative  $\partial_x^i \partial^j_y \phi_{\sa \mu} $, with $i+j\le r$, we use the  Faa-di-Bruno formula to compute $\partial_x^i  \phi_{\sa \mu} $, and show that this is equivalent to prove that the following contracts the uniform norm of the $j^{th}$-derivative $\partial_y^j$, for $2\le i\le r-j$:
\[\partial_x^i \phi_\sb \mapsto   \lambda_{\st \mu}^{-1}\cdot  \partial^i_x \phi_{\sb}(\rho_{\st \mu}(x), g_{\st\mu}(x,y))
(\rho_{\st \mu}'(x))^i\]
For the same reason this operator is $\lambda_{\st} \sigma_{\st}^i$-contracting (for an equivalent norm), since by $(\mathcal D_r)$ of remark \ref{rema2.9}, we can assume that for equivalent Riemannian metric that $\lambda_{\st} \sigma_{\st}^i<1$,  $2\le i\le r-j$.

We lose one more derivative in $x$ since $[2, r-i]$ is empty for $i=r-1$.

 Let us denote by  $((\phi_{\sa\mu})_\sa)_{\mu}$ the fixed point of $(\Phi_\mu)_\mu$. Then we consider the family of charts systems  $(\mathcal Y'(\mu))_\mu$, with $\mathcal Y'(\mu)=(y'_\sa)_{\sa\in \sY}$ given by:
\[y'_\sa\colon (x,y)\mapsto y_\sa(x, \phi_{\sa\mu}^{-1}(x,y)).\]
We have for every $\st=(\sa,\sb)\in \sT$:
\[y_\sb'^{-1}\circ f\circ y_{\sa}'(x,y)= (\sigma_\st x, \phi_{\sb\mu}^{-1}(\sigma_\st\cdot x,  g_{\st\mu}(x,\phi_{\sa\mu}(x,y)))= (\sigma_\st\cdot  x, \lambda_{\st}\cdot y).\]

 The $C^{r-2}$-smoothness with respect to the parameter is shown following the same argument as in Proposition \ref{semilinearization}.

\begin{proof}[Proof of Lemma \ref{Lemmadederive}] We recall that the system of charts is of type B. 
 For every $\st=(\sa,\sb)\in \sT$ and $\mu\in \M$, for every $x\in I^u_{\sa}\cap (I^u_{\sb}/\sigma_\st)$, we put:
\[\delta_{\st\mu}(x):= \frac{\partial_y g_{\st\mu}(x,0)}{\lambda_{\st\mu}},\quad \text{with }{\lambda_{\st\mu}:=\partial_y g_{\st\mu}(0,0)}.\]

We remark that the map $(x,\mu)\mapsto \delta_{\st\mu}(x)$ has its derivatives $\partial^i_x \partial_\mu^k$ which are continuous and bounded for all $i+k\le r-1$ with $k\not= r-1$.    

For ever $\sa\in \sY$, if $(\sa_i)_{-n\le i \le 0}$ is a maximal admissible chain with $\sa_0=\sa$ and $n\in \N\cup\{\infty\}$, the function is well defined:
\[\Delta_{\sa\mu}(x):= \prod_{i=-n+1}^{0} \delta_{\st_{i}\mu}\big(\frac{x}{\sigma_{\st_{i\mu }}\cdots \sigma_{\st_{0}\mu}}\big),\quad \text{with }\st_i:=  \sa_{i-1}\sa_{i}.\]

From property $(\mathcal M)$, the map $(x,\mu)\mapsto \Delta_{\sa \mu}(x)$ has its derivative $\partial^i_x \partial_\mu^k$ which are continuous and bounded for all $i+k\le r-1$ with $k\not= r-1$.

Then we consider the family of charts systems  $(\mathcal Y'(\mu))_\mu$, with $\mathcal Y'(\mu)=(y'_\sa)_{\sa\in \sY}$ given by:
\[y'_\sa\colon (x,y)\mapsto y_\sa(x, \Delta_{\sa\mu}(x)\cdot y).\]
We have for every $\st=(\sa,\sb)\in \sT$:
\[y_\sb'^{-1}\circ f_\mu\circ y_{\sa}'(x,y)= (\sigma_\st x, \Delta_{\sb\mu}(\sigma_{\st\mu}\cdot x)^{-1}\cdot g_{\st\mu}(x,\Delta_{\sa\mu}(x)\cdot y)).\]
We notice that it is here crucial for $\Delta_{\sa\mu}$ to be bounded away of zero in order that the new system of charts and its inverse have 
bounded and continuous derivatives $\partial^i_x\partial^j_y \partial_\mu^k$, for all $i+j+k\le r-1$  with $k\not= r-1$.
Here we used that $f_\mu$ is Hénon like.

For $y=0$, it holds with $\lambda_\st := \partial_y g_{\st \mu}(0,0)=D_0h_{\st \mu}$: 
\[\partial_y (\Delta_{\sb\mu}(\sigma_{\st \mu} x)^{-1}\cdot g_{\st\mu}(x,\Delta_{\sa\mu}(x)\cdot y))|\{y=0\}
= \partial_y g_{\st\mu}(x,0)\cdot \frac{\Delta_{\sa\mu}(x)}{\Delta_{\sb\mu}(\sigma_{\st \mu} x)}=\frac{\partial_y g_{\st\mu}(x,0)}{\delta_{\st\mu}(x)}=\lambda_{\st\mu}.\]
\end{proof}
\section{Computational proof}
\subsection{Proof of Proposition \ref{PTestimate}}
\label{proofofpropPTestimate}
We recall that $DC_\mu\circ L^n_\mu(x,y)= (\sigma_\mu^{2n} x, \lambda_\mu^n \sigma_\mu^{2n} y)$ and
\[(0, (b-b'\lambda^n_\mu \sigma^n_\mu )x)+ DC_\mu\circ L^n_\mu \circ E'_\mu\circ C^{-1}_\mu=:(A_\mu,b\cdot B_\mu).\]

{\bf $C^1$ norm.}
To evaluate the $C^1$-norm of $A_0$ and $B_0$, it is sufficient to show that their values  at $(0,-\sigma^{2n}r)$ are small since we will show that their second derivatives are small.

Observe that $C_\mu^{-1}(0,-\sigma_\mu^{2n}r)=P$. As $E'_0(P)$ and $DA'_0(P)=0$ and $\partial_x B_0'(P)=0$ by $(\mathcal E_2)$ and $(\mathcal E_3)$, it holds:
\[ (A_0, b\cdot B_0)(0,-\sigma_0^{2n}r)=0,\quad  DA_0(0,-\sigma_0^{2n}r)=0,\quad D A_0(0,-\sigma^{2n} r)=0,\quad \partial_x B_0(0,-\sigma^{2n} r)=0\]
Also $\partial_y  B_0 = \lambda_0^{n} (\partial_y B_0')\circ C_0^{-1}$ is small for $n$ large.
 
{\bf $C^2$ norm.}
Also we recall  that $DC_\mu\circ L^n_\mu$ is linear, thus it is equal to its derivative, and its second derivative is 0.  The norm of $DC^{-1}_\mu$ is less than $\sigma^{-n}_\mu$, and so:
\[D^2 A_0= \sigma^{2n}_\mu D^2 A'_\mu(DC_\mu^{-1},DC_\mu^{-1})\quad \text{and}\quad  D^2 B_\mu= \sigma^n_\mu D^2 B'_\mu(DC_\mu^{-1},DC_\mu^{-1})\]
$D^2 B_\mu$ is dominated by $\sigma_\mu^{-n}$ which is small.

The same holds for $\partial_x\partial_y A_\mu$ and $\partial_{y}^2 A_\mu$. On the other hand, by $(\mathcal E_3)$, $\partial_x^2 A_\mu(0, -\sigma_\mu^{2n} r)=0$,  
thus $\partial_x^2 A_\mu(P')=0$ and so $\partial_x^2 A_\mu(z)$ is small when $C^{-1}_\mu(z)$ is close to $P$, which is the case whenever $\mu=M_n(a)$, and $(z,a)$ are in a compact set, while $n$ is large. 

{\bf $C^r$ norm for $r\ge 3$.} Similarly, we have:
\[D^r A_\mu= \sigma^{2n}_\mu D^r A'_\mu(DC_\mu^{-1})^r\quad \text{and}\quad  D^r B_\mu= \sigma^n_\mu D^r B'_\mu(DC_\mu^{-1})^r\]
Hence $D^r B_\mu$ is dominated by $\sigma_\mu^{(r-1)n}$ and $D^r A_\mu$ is dominated by $\sigma_\mu^{(r-2)n}$ which are both small.

{\bf Parameter dependence.} We compute:
\[\partial_a A_a= \partial_a \sigma_\mu^{2n} A'_\mu\circ C_\mu^{-1} +\sigma_\mu^{2n} \partial_a (A'_\mu\circ C_\mu^{-1})\]
\[\partial_a B_a= \frac{ \partial_a (\lambda_\mu^n\cdot \sigma_\mu^{2n}) B'_\mu\circ C_\mu^{-1} +\lambda_\mu^n\cdot \sigma_\mu^{2n} \partial_a (B'_\mu\circ C_\mu^{-1})}{\lambda_0^n\cdot \sigma_0^{n}}
+\partial_ a \left[\left(\left(1- \frac{\sigma_\mu \lambda_\mu}{\sigma_0 \lambda_0}\right)^n\right)x\right]
\]
 Observe that the derivative of $\mu$ w.r.t $a$ is $\xi^{-1} \sigma_\mu^{-2n}$, and that its second derivative is 0. Thus 
\begin{equation}\label{dA}\partial_a A_a=\xi^{-1} \sigma_\mu^{-2n} (\partial_\mu \sigma_\mu^{2n}) A'_\mu \circ C_\mu^{-1} +\xi^{-1} \partial_\mu (A'_\mu )\circ C_\mu^{-1}+\xi^{-1} DA'_\mu \circ \partial_\mu (C_\mu ^{-1})\end{equation}
\begin{equation}\label{dB}
\partial_a B_a=
  \frac{\partial_\mu  (\sigma^n_\mu \lambda^n_\mu)}{\sigma^n_0 \lambda^n_0 \sigma^{2n}_\mu \xi }x +
\frac{ \partial_\mu (\lambda_\mu^n\cdot \sigma_\mu^{2n})}{\xi \sigma_\mu^{2n} \lambda_0^n\cdot \sigma_0^{n}} B'_\mu\circ C_\mu^{-1}
 +\xi^{-1} \frac{\lambda_\mu^n}{\lambda_0^n\cdot \sigma_0^{n}}\big(\partial_\mu (B'_\mu)\circ C_\mu^{-1}+DB'_\mu\circ \partial_\mu (C_\mu^{-1})\big)\end{equation}

\noindent\underline{Bounds for (\ref{dA})}.

Note that the following is dominated by $n$:
\[\sigma_\mu^{-2n} (\partial_\mu \sigma_\mu^{2n})=\partial_\mu \log(\sigma_\mu^{2n})= 2n \partial_\mu \log(\sigma_\mu)\] 
The same holds for higher derivatives with respect to $\mu$. 

We remark also that $(x,y,\mu)\mapsto C_\mu^{-1}(x,y)= (\sigma_\mu^{-n}x,\sigma_\mu^{-2n} y)$ is $C^{r}$-exponentially small for $n$ large: 
its Hessian is 0, its derivatives with respect to $\mu$ is of the order of $n\sigma^{-n}_\mu$.

For $a$ in a compact set, the image by $C_\mu^{-1}$ of any ball is an exponentially small neighborhood of $P$, also $\mu$ is exponentially small when $n$ is large. Thus, by $(\mathcal E_2)$ and $(\mathcal E_3)$, the $C^0$-norm of $A'_\mu\circ C_\mu^{-1}$ and $\partial_\mu A'_\mu\circ C_\mu^{-1}$ restricted to any compact subset is exponentially small when $n$ is large.  Consequently, the $C^0$-norm of $(x,y)\mapsto \partial_a A_a(x,y)$ is small. Also $DC_\mu^{-1}$ and $D\partial_a C_\mu^{-1}$ are exponentially small, the $C^1$-norm of $\partial_a A_a$ is exponentially small. Moreover since the higher derivatives of $A'_\mu$ are bounded, it comes as in the study without parameter dependence that the $C^{r-1}$-norm of $(x,y)\mapsto \partial_a A_a(x,y)$ is exponentially small. Similarly,  
for every $k< r$, the $C^{r-k-1}$-norm of $\partial_\mu^{k-1} \partial_a A_a$ is bounded. Therefore, the norm  $C^{r-k}$-norm of $\partial_a^k A_a$ is small.

\noindent\underline{Bounds for (\ref{dB})}.
First note that, since $|\lambda_\mu /(\sigma_\mu\cdot \lambda_0\cdot \sigma_0)|$ is smaller than one, the following is exponentially small:
\[\frac{ \partial_\mu (\lambda_\mu^n\cdot \sigma_\mu^{2n})}{\sigma_\mu^{2n}\lambda_0^n\cdot \sigma_0^{n}}=
\sigma_\mu^{-2n}\frac{ \partial_\mu (\lambda_\mu^n\cdot \sigma_\mu^{2n})}{\lambda_\mu^n\cdot \sigma_\mu^{n}}\frac{\lambda_\mu^n\cdot \sigma_\mu^{n}}{\lambda_0^n\cdot \sigma_0^{n}}=2n \left(\sigma_\mu^{-2}  \frac{\lambda_\mu \cdot \sigma_\mu}{\lambda_0\cdot \sigma_0}\right)^n\partial_\mu \log (\lambda_\mu\cdot \sigma_\mu^{2})\]
The same occurs for higher derivatives with respect to $\mu$ (and so w.r.t. $a$). 

Likewise $ \mu \mapsto \frac{\partial_\mu  (\sigma^n_\mu \lambda^n_\mu)}{\sigma^n_0 \lambda^n_0 \sigma^{2n}_\mu \xi }$ and $\mu\mapsto \frac{\lambda_\mu^n}{\lambda_0^n \sigma_0^n}$ are $C^{r-1}$-exponentially small for $n$ large. 

Hence,  using $(\mathcal E_2)$ as above, the $C^0$-norm of $A'_\mu\circ C_\mu^{-1}$ restricted to any compact subset is exponentially small when $n$ is large. 
Furthermore, since the $C^r$-norm of $(x,y,a)\mapsto B_a (x,y)$ is bounded, it comes that the $C^r$-norm of $(x,y,\mu)\mapsto B'_\mu (x,y)$ is $C^r$-small.

\bibliographystyle{alpha}
\bibliography{references}
\end{document}